\documentclass{article}

\usepackage[margin=1in,a4paper]{geometry}
\usepackage{amssymb,amsmath,,wasysym}
\usepackage{amsthm} 
\usepackage{tikz-cd}
\usepackage{xcolor}
\usepackage{float}
\usepackage[colorlinks,linkcolor=blue,citecolor=blue,pagebackref,hypertexnames=false, breaklinks]{hyperref}
\usepackage{circuitikz}
\usepackage[shortlabels]{enumitem}
\setlist[enumerate,1]{wide, labelindent=0pt,label={\upshape(\roman*)}}

    \def\XXint#1#2#3{{\setbox0=\hbox{$#1{#2#3}{\int}$}
    \vcenter{\hbox{$#2#3$}}\kern-.5\wd0}}

\def\vint{\mathop{\mathchoice%
          {\setbox0\hbox{$\displaystyle\intop$}\kern 0.22\wd0%
           \vcenter{\hrule width 0.6\wd0}\kern -0.82\wd0}%
          {\setbox0\hbox{$\textstyle\intop$}\kern 0.2\wd0%
           \vcenter{\hrule width 0.6\wd0}\kern -0.8\wd0}%
          {\setbox0\hbox{$\scriptstyle\intop$}\kern 0.2\wd0%
           \vcenter{\hrule width 0.6\wd0}\kern -0.8\wd0}%
          {\setbox0\hbox{$\scriptscriptstyle\intop$}\kern 0.2\wd0%
           \vcenter{\hrule width 0.6\wd0}\kern -0.8\wd0}}%
          \mathopen{}\int}

\newcommand{\vast}{\bBigg@{4}}
\newcommand{\Vast}{\bBigg@{5}}
\makeatother
\newcommand{\arctanh}{\mathrm{arctanh}}
\newcommand{\arcosh}{\mathrm{arcosh}}

\newcommand{\R}{\mathbb{R}}

\newcommand{\bS}{\mathbb{S}}

\theoremstyle{plain}
\newtheorem{thm}{Theorem}[section]
\newtheorem{theorem}[thm]{Theorem}

\newtheorem{lemma}[thm]{Lemma}

\newtheorem{definition}[thm]{Definition}
\newtheorem{corollary}[thm]{Corollary}

\newtheorem{proposition}[thm]{Proposition}

\newtheorem{example}[thm]{Example}

\theoremstyle{definition}

\theoremstyle{remark}
\newtheorem{remark}[thm]{Remark}
\newcommand{\bremark}{\begin{remark} \em}
\newcommand{\eremark}{\end{remark} }
\usepackage{color}

\hbadness=\maxdimen

\usepackage[textwidth=3.2cm]{todonotes}

\begin{document}

\title{On the law of  the index of Brownian loops related to the Hopf and anti-de Sitter fibrations}

\author{Fabrice Baudoin\footnote{fbaudoin@math.au.dk}, Teije Kuijper\footnote{t.kuijper@math.au.dk} \\
\\ \textit{Department of Mathematics, Aarhus University}}

\date{}
\maketitle

\begin{abstract}
We give explicit formulas and asymptotics for the distribution of the index of the Brownian loop in the following geometrical settings: the complex projective line from which two points have been removed; the complex hyperbolic line from which one point has been removed; the odd dimensional spheres from which a great hypersphere has been removed; and the complex anti-de Sitter spaces.
Our analysis is based on the geometry of the Hopf and anti-de Sitter fibrations, and on the relationship between winding and area forms.
 \end{abstract}


\section{Introduction}

In a celebrated 1980's paper  \cite{Yor1980LoiDL} M. Yor obtained an exact formula for the law of the index of the Brownian loop in $\mathbb{C}$. Namely, let us denote by $(X^{z_0}(t))_{0\le t \le L}$ a Brownian motion loop started from $z_0 \neq 0$ and conditioned to go back to $z_0$ at time $L>0$.
Since the point $ 0 $ is polar one can find a continuous process $(\theta(t))_{0 \le t \le L}$ such that 
\[
X^{z_0}(t) = | X^{z_0}(t) | e^{i \theta(t)} ,
\]
for every $t  \in [0,L]$.
The index of the Brownian loop is the discrete random variable $\frac{1}{2\pi} (\theta(L)-\theta(0))$ which takes its values in $\mathbb Z$.
The distribution of this random variable as computed in \cite[Theorem 6.10]{Yor1980LoiDL} is  given by
\begin{align}\label{Yor formula}
\mathbb{P} \left( \theta(L)-\theta(0)=2\pi k \right) =e^{-r} \left( \Phi_r((2k-1)\pi) -\Phi_r((2k+1)\pi) \right)\textrm{ for } k \in \mathbb Z \setminus \{ 0\} ,
\end{align}
where $r:= |z(0)|^2/L$ and 
\[
\Phi_r(x):=\frac{x}{\pi} \int_0^{+\infty} e^{-r \cosh (t)} \frac{dt}{t^2+x^2}\textrm{ for } x \neq 0.
\]
For $k=0$, we have  $\mathbb{P} \left( \theta(L)-\theta(0)=0 \right)=1-2e^{-r}\Phi_r(\pi)$.
By looking at the Fourier transform of the law of the index it can be shown that we recover Spizer's theorem, but then for the Brownian loop, more precisely
\begin{align*}
    \frac{2\theta (L)}{\ln (L)}\rightarrow C_1
\end{align*}
in distribution, as $L\rightarrow\infty$.
Furthermore the law of the index has Cauchy tails, that is
\begin{align*}
    \mathbb{P}(\theta (L)-\theta (0)=2\pi k)\sim_{|k|\rightarrow +\infty}\frac{C(L)}{k^2}
\end{align*}
for some constant $C(L)>0$, see \cite[Corollary  5.21]{Yor1980LoiDL}. We refer to the nice survey \cite{LeGall} by J.F. Le Gall for ramifications and generalizations of the formula \eqref{Yor formula} which illustrate its importance.
We note that the index of the Brownian loop in the plane can be written as a stochastic line integral
\[
\theta(L)-\theta(0) =\int_{X^{z_0}[0,L]} \alpha
\]
where $\alpha=\frac{xdy-ydx}{x^2+y^2}$ is a closed but non exact one-form  which can be called the winding form in $\mathbb{C}\setminus \{0 \}$.  This form $\alpha$ is intrinsic (up to a sign) in the sense that the equivalence class $\frac{1}{2\pi} [\alpha]$ is a generator of the integral de Rham cohomology group $H_1^{dR}(\mathbb{C}\setminus \{0 \}, \mathbb{Z}) \simeq \mathbb{Z}$.

Let $\mathbb M$ be a Riemannian manifold, its integral first de Rham  cohomology group $H_1^{dR}(\mathbb{M}, \mathbb{Z})$ is defined as the set of de Rham equivalence classes of one-forms whose integrals along smooth loops are integers.
If  $H_1^{dR}(\mathbb{M}, \mathbb{Z}) \simeq \mathbb Z$ we say that $\alpha$ is a \emph{winding} form if it  is a generator of $H_1^{dR}(\mathbb{M}, \mathbb{Z})$.  If $\gamma$ is a loop in $\mathbb M$ we define its index by
\[
\mathbf{Ind} (\gamma) :=\int_\gamma \alpha \in \mathbb{Z}.
\]
Note that this index is uniquely defined up to sign. One of our goals in this paper is to compute the exact distribution of the random variable $\mathbf{Ind} (\gamma)$, where $\gamma$ is a Brownian loop, for several different Riemannian manifolds $\mathbb{M}$. To illustrate our results with a simple example, consider for $\mathbb M$ the Lie group $\mathbf{SL}(2,\mathbb R)$ of $2 \times 2$ real unimodular matrices. Since $\mathbf{SL}(2,\mathbb R)$ is diffeomorphic to $\mathbb{R}^2 \times \mathbb S^1$, one has $H_1^{dR}(\mathbf{SL}(2,\mathbb R), \mathbb{Z}) \simeq \mathbb{Z}$. In that case, we prove the following theorem, see Theorem \ref{SL2 case} (with $\mu=1$):

\begin{theorem}\label{theorem intro}
Let $(X(t))_{0 \le t \le L}$ be a  Brownian loop of length $L$ in $\mathbf{SL}(2,\mathbb R)$, i.e. a Brownian motion started  from the identity and conditioned to come back to identity at time $L$.
Then, for every $k \in \mathbb Z$,
\[
\mathbb{P}\left(  \mathbf{Ind}(X[0,L]) =k \right) =C \, e^{-\frac{\pi^2k^2}{L}}  \int_{-\infty}^{+\infty}\cos \left(\frac{\pi k y}{L}  \right) \frac{y}{\sinh(y)}e^{-\frac{y^2}{4L} }dy,
\]
where $C >0$ is the normalization constant. In particular, in distribution, when $L\to +\infty$,
\[
\frac{2\pi\mathbf{Ind}(X[0,L])}{\sqrt{2L}} \to \mathcal{N}(0,1)
\]
where $\mathcal{N}(0,1)$ is a normal random variable with mean zero and variance one.
\end{theorem}

To prove Theorem \ref{theorem intro} the main strategy is to consider the anti-de Sitter fibration 
\[
\mathbb{S}^1 \to \mathbf{SL}(2,\mathbb{R}) \to \mathbb{C}H^1,
\]
and to realize that a winding form $\alpha$ is given by $\frac{1}{2\pi} (\eta -\omega)$, where $\eta$ is the contact form on $\mathbf{SL}(2,\mathbb R)$ and $\omega$ the pull-back to $\mathbf{SL}(2,\mathbb R)$ of the area form on $\mathbb{C}H^1$.
Therefore, if $(\beta(t))_{ t\ge 0}$ is a Brownian motion on $\mathbf{SL}(2,\mathbb R)$, then
\[
\int_{\beta[0,t]} \alpha=\frac{1}{2\pi} \int_{\beta[0,t]} \eta - \frac{1}{2\pi} \int_{\beta[0,t]} \omega.
\]
It turns out that $ \int_{\beta[0,t]} \eta$ is a one-dimensional Brownian motion independent from the projection of $(\beta(t))_{ t\ge 0}$ onto $\mathbb{C}H^1$ and that $\int_{\beta[0,t]} \omega$ is the stochastic area functional on $\mathbb{C}H^1$ studied in \cite[Section 4.2]{Baudoin2022} and \cite{MR3719061}.
Taking advantage of the fact that the stochastic area functional was already studied one can get formulas, which after computations combined with new ideas and  techniques borrowed from \cite{Yor1980LoiDL}, yield Theorem \ref{theorem intro}.

More generally, in this paper we focus on distributions related to the index of the Brownian loop in the following examples:
\begin{enumerate}
\item $\mathbb{M}=\mathbb{C}P^1 \setminus \{ p,q \}$ where $\mathbb{C}P^1$ is the complex hyperbolic space of real dimension 2 and $p,q$ are diametrally distant points on it;
\item $\mathbb{M}=\mathbb{S}^{2n+1} \setminus \mathcal{C}$ where $\mathbb{S}^{2n+1}$  is a $2n+1$ dimensional sphere and $\mathcal{C}$ a $2n-1$ dimensional great sphere included in $\mathbb{S}^{2n+1}$;
\item $\mathbb{M}=\mathbb{C}H^1 \setminus \{ p \}$ where $\mathbb{C}H^1$ is the complex hyperbolic space of real dimension 2 and $p$ a point in it;
\item $\mathbb{M}=\mathbf{AdS^{2n+1}}$ where $\mathbf{AdS^{2n+1}}$ is the $2n+1$ dimensional complex anti-de Sitter space.
\end{enumerate}

For the examples (ii) and  (iv) the geometric idea is similar to the one for $\mathbf{SL}(2,\mathbb{R})$ which is explained above: The key point is the relation between the winding form $\mathbb M$ and the area form on a quotient space. For all the examples, the main ingredient for explicit computations of characteristic functions is the Yor's transform \cite[Section A.9]{Baudoin2022} which was first introduced in \cite{Yor1980LoiDL}. 
For added generality, we will consider windings of general Brownian bridges rather than only Brownian loops and for the examples (ii) and  (iv) we will even introduce a one-parameter variation of the Brownian motion for which computations are possible.
Geometrically, such a variation is the family of Brownian motions corresponding to the so-called canonical variation of the Riemannian metric of a foliation as in \cite[Chapter 9.G]{Besse2007einstein}.

The paper is organized as follows:
Section 2 deals with windings of Brownian bridges in the complex projective space $\mathbb{C}P^1$ and the odd-dimensional spheres;
Section 3 with  windings of Brownian bridges in the complex hyperbolic space $\mathbb{C}H^1$ and the anti-de Sitter spaces.

\section{Windings on the Hopf fibration} 

\subsection{Preliminaries}

We first start with some preliminaries on some special functions and reminders about Brownian motions on complex projective spaces which will be used throughout this section. Throughout the text we use special functions and some of their basic properties, and we refer to the classical book \cite{special} for a detailed account.

\subsubsection{Jacobi diffusions}

 A Jacobi diffusion  is a diffusion on the interval $[0,\pi/2]$ with generator
\[
\mathcal{L}^{\alpha,\beta}:=\frac{1}{2} \frac{\partial^2}{\partial r^2}+\left(\left(\alpha+\frac{1}{2}\right)\cot (r)-\left(\beta+\frac{1}{2}\right) \tan (r)\right)\frac{\partial}{\partial r}\textrm{ for } \alpha,\beta >-1,
\]
which is defined up to the first time it hits the boundary $\{0, \pi/2\}$.
Let $q_t^{\alpha,\beta}(r(0),\cdot)$ denote the transition density, with respect to the Lebesgue measure on $[0,\pi/2]$, of a Jacobi diffusion   started from $r(0)$.
From \cite[Section B.2]{Baudoin2022} it is known that for $\alpha,\beta \ge 0$, $t>0$ and $r \in (0,\pi/2)$ we have

\begin{align}\label{jacobi heat kernel}
 q^{\alpha,\beta}_t(r(0),r)=2(\cos (r))^{2\beta +1} (\sin (r))^{2\alpha +1}\sum_{m=0}^{+\infty} &(2m+\alpha+\beta+1)\frac{\Gamma(m+\alpha+\beta+1)\Gamma(m+1)}{\Gamma(m+\alpha+1)\Gamma(m+\beta+1)}  \\
  &e^{-2m(m+\alpha+\beta+1)t} P_m^{\alpha,\beta}(\cos (2r(0)))P_m^{\alpha,\beta}(\cos (2r)) \notag,
\end{align}
where
\[
P_m^{\alpha,\beta}(x)=\frac{(-1)^m}{2^mm!(1-x)^{\alpha}(1+x)^\beta}\frac{d^m}{dx^m}((1-x)^{\alpha+m}(1+x)^{\beta+m})
\]
is a Jacobi polynomial.
Let us notice the normalization
\[
P_m^{\alpha,\beta}(1)=\frac{\Gamma(m+\alpha+1)}{\Gamma(m+1)\Gamma(\alpha+1)}.
\]
When $\alpha=\beta \ge 0$,   the Jacobi polynomial $P_m^{\alpha,\alpha}$ reduces to the Gegenbauer polynomial:
\begin{align}\label{eq:Gegenbauer_polynomial_expression}
C_m^{\alpha+1/2}(\cos(2r))=\frac{\Gamma(\alpha+1)\Gamma(m+2\alpha+1)}{\Gamma(2\alpha+1)\Gamma(m+\alpha+1)} P_m^{\alpha,\alpha}(\cos (2r)),
\end{align}
in particular
 \[
 C_m^{\alpha+1/2}(1)=\frac{\Gamma(2\alpha+1+m)}{\Gamma(2\alpha+1)\Gamma(m+1)}.
 \]
From \cite[Theorem 6.7.4]{special} an integral representation of the Gegebauer polynomials is given by
\begin{align}\label{representation Gegenbauer}
  \frac{\Gamma(\alpha+1/2)\sqrt{\pi}}{\Gamma(\alpha+1)}\frac{C_m^{\alpha+1/2}(\cos 2r)}{C_m^{\alpha+1/2}(1)}=\int_0^\pi \left( \cos (2r)+i \sin (2r) \cos (\eta) \right)^m (\sin (\eta))^{2\alpha} d\eta.
\end{align}

\subsubsection{Brownian motion on \texorpdfstring{$\mathbb{C}P^n$}{Cpn}}\label{BM CPn}

The complex projective space\index{projective space!complex } $\mathbb{C}P^n$ can be defined as the set of complex lines in $\mathbb{C}^{n+1}$. To parametrize points in $\mathbb{C}P^n$, it is convenient to use the local affine coordinates given by $w_j=z_j/z_{n+1}$ for $1 \le j \le n$, $z \in \mathbb{C}^{n+1}$, and $z_{n+1}\neq 0$.
In these coordinates, the Riemannian structure of $\mathbb{C}P^n$ is easily worked out from the standard Riemannian structure of the Euclidean sphere.
More precisely, let us consider the unit sphere
\[
\bS^{2n+1} :=\bigg\lbrace z=(z_1,\dots,z_{n+1})\in \mathbb{C}^{n+1} \bigg| \, | z |^2:=\sum_{i=1}^{n+1}|z_i|^2 =1\bigg\rbrace,
\]
and note that the map $\bS^{2n+1} \setminus\{z_{n+1}=0 \} \to \mathbb{C}^n$, $ (z_1,\dots,z_{n+1}) \to (z_1/z_{n+1},\cdots,z_n/z_{n+1})$ is a submersion. $\mathbb{C}^n$ is therefore diffeomorphic to a dense open subset of $\mathbb{C}P^n$. The unique Riemannian metric on $\mathbb{C}^n$ that makes the above submersion a Riemannian submersion induces the so-called \emph{Fubini-Study metric} \index{Fubini Study metric} on $\mathbb{C}P^n$. Explicitly, in affine coordinates, the Fubini-Study Hermitian metric $h$ on $\mathbb{C}P^n$ satisfies
\[
h \left(\frac{\partial}{\partial w_k}, \frac{\partial}{\partial \overline{w}_j} \right)= \frac{(1+|w|^2) \delta_{kj} -\overline{w}_k w_j}{(1+|w|^2)^2}\textrm{ for } 1\le k,j\le n,
\]
and the Laplace-Beltrami operator on $\mathbb{C}P^n$ is given by:

\begin{equation}\label{eq-Laplacian-CPn1}
\Delta_{\mathbb{C}P^n}=4(1+|w|^2)\sum_{k=1}^n \frac{\partial^2}{\partial w_k \partial\overline{w}_k}+ 4(1+|w|^2)\mathcal{R} \overline{\mathcal{R}},
\end{equation}
where $|w|^2 :=\sum_{k=1}^n | w_k|^2$ and
\[
\mathcal{R} :=\sum_{j=1}^n w_j \frac{\partial}{\partial w_j}.
\] 
The Riemannian volume measure is given 
\[
\frac{dw}{(1+|w|^2)^{n+1}},
\]
where $dw$ denotes the Lebesgue measure on $\mathbb{C}^n$.
The Riemannian distance on $\mathbb{C}P^n$ between two points of local affine coordinates $w$ and $\tilde{w}$ will be denoted by $d(w,\tilde{w})$.
One can check that if a point has coordinate $w$, then
\[
d(0,w)=\arctan (|w|).
\]

Brownian motion on $\mathbb{C}P^n$ is a diffusion with generator $\frac{1}{2}\Delta_{\mathbb{C}P^n}$.
The radial part of Brownian motion on $\mathbb{C}P^n$ is a Jacobi diffusion.
More precisely, let us consider Brownian motion $(w(t))_{t \ge 0}$ on $\mathbb{C}P^n$ issued from a point $w(0) $.
The process $r(t):=d(0,w(t))=\arctan (|w(t)|)$ is a Jacobi diffusion with generator $\mathcal{L}^{n-1,0}$ from which it follows that  the density of $w(t)$ with respect to the Lebesgue measure is given at a point $w$ by the formula
\begin{align}\label{heat kernel CPn}
p_{\mathbb{C}P^n}(t,w(0), w)= \frac{\Gamma \left(n \right)}{2 \pi^{n}}  \frac{q_t^{n-1,0}(0,d(w(0),w))}{\cos(d(w(0),w)) \sin(d(w(0),w))^{2n-1} } \frac{1}{(1+|w|^2)^{n+1}},
\end{align}
for $t>0$.
When $w=w(0)$, we of course understand \eqref{heat kernel CPn} as the limit $w \to w(0)$.
Explicitly, thanks to formula \eqref{jacobi heat kernel}, we have
\begin{align*}
p_{\mathbb{C}P^n}(t,w(0), w(0))= \frac{\Gamma \left(n \right)}{ \pi^{n}}  \left( \sum_{m=0}^{+\infty} (2m+n) \frac{\Gamma\left(m+n\right)^2}{\Gamma\left(m+1\right)^2\Gamma\left(n\right)^2} e^{-2m(m+1)t} \right) \frac{1}{(1+|w(0)|^2)^{n+1}}. 
\end{align*}

\subsection{Windings on \texorpdfstring{$\mathbb{C}P^1$}{CP1}}

In this section we will look at the winding number of the Brownian motion $(w(t))_{t \ge 0}$ on $\mathbb{C}P^1$ around the north and south pole.\footnote{In homogeneous coordinates, the north pole is given by $[0:1]$ and the south pole by $[1:0]$.}
We will assume that $(w(t))_{t \ge 0}$ is started from a point $w(0)$ which is neither the north pole nor the south pole.
The north pole $0$ is polar for the Brownian motion $(w(t))_{t\ge 0}$, one can therefore write a polar decomposition
\[
w(t) =| w(t)| e^{i\theta(t)},
\]
where 
\[
\theta(t)=\theta(0)+\frac{1}{2i}\int_0^t \frac{ \overline{w}(t)dw(t)-d\overline{w}(t)w(t)}{|w(t)|^2}
\]
and $\theta(0)$ is such that $w(0)=|w(0)|e^{i\theta(0)}$.
Using rotational invariance, we can assume, without loss of generality, that $\theta(0)=0$.
Consequently, by the above and since we are working in conformal coordinates, it is natural to make the following definition:
\begin{definition}\index{winding! process}
The winding process of the Brownian motion $w(t)$ is defined by 
\begin{equation*}
\theta (t) :=  \frac{1}{2i}\int_0^t \frac{ \overline{w}(s)dw(s)-d\overline{w}(s)w(s)}{|w(s)|^2}\textrm{ for } t \geq 0,
\end{equation*}
where $(w(t))_{t \geq 0}$ is the Brownian motion on $\mathbb{C}P^1$. 
\end{definition}

From \eqref{eq-Laplacian-CPn1} the generator of the Brownian motion on $\mathbb{C}P^1$ in spherical coordinates $(r,\phi)$ given by $w=\tan(r)e^{i\phi}$, reads:
\[
\frac{1}{2} \left( \frac{\partial^2}{\partial r^2} + 2 \cot (2r) \frac{\partial}{\partial r} +\frac{4}{\sin^2 (2r)} \frac{\partial^2}{\partial \phi^2}\right)\textrm{ for } r\in[0,\pi/2]\textrm{ and }\phi \in \mathbb{R},
\]
where $r$ parametrizes the Riemannian distance from 0 in $\mathbb{C}P^1$. This shows that the winding process of Brownian motion in $\mathbb{C}P^1$ is given by
\begin{align}\label{ClockWinding}
\theta(t)=\beta\left(\int_0^t \frac{4ds}{\sin^2 (2r(s))}\right)\textrm{ for } t \geq 0,
\end{align}
where $(r(t))_{t\geq 0}$ is a Jacobi diffusion started at $r(0)\in (0,\pi /2)$ with generator 
\[
\frac{1}{2} \left( \frac{\partial^2}{\partial r^2} + 2 \cot (2r) \frac{\partial}{\partial r} \right)
\]
and $(\beta (t))_{t\geq 0}$ is an one-dimensional Brownian motion independent from $(r(t))_{t\geq 0}$ started from 0.

\begin{proposition}\label{prop:Fourier(t)ransform_CP}
For $t >0$ and $\lambda \in \mathbb R$, we have 
\begin{align}\label{fourier_conditional_radial_CP1}
\mathbb{E} \left( e^{i \lambda \theta (t)}  \mid r(t)\right)=e^{-2|\lambda|(|\lambda|+1)t}\left(\frac{\sin(2r(0))}{\sin(2r(t))}\right)^{|\lambda|}  \frac{q_t^{|\lambda|,|\lambda|}(r(0),r(t))}{q_t^{0,0}(r(0),r(t))}.
\end{align}
\end{proposition}

\begin{proof}
We use the Yor transform method explained in all generality in \cite[Section A.9]{Baudoin2022}.
Using symmetry, we can assume $\lambda >0$.
From \eqref{ClockWinding} and the independence of $(r(t))_{t\geq 0}$ and $(\beta (t))_{t\geq 0}$ we have, for every bounded Borel function $f$, that
\[
\mathbb{E}\left( e^{i\lambda  \theta (t)} f(r(t))\right)=\mathbb{E}\left( e^{-\frac{\lambda^2}{2} \int_0^t \frac{4ds}{\sin^2 (2r(s))}} f(r(t))\right)=e^{-2\lambda^2t} \mathbb{E}\left( e^{-2\lambda^2 \int_0^t \cot^2 (2r(s)) ds} f(r(t))\right).
\]
We now observe that the process $r$ is the unique strong solution of a stochastic differential equation
\[
r(t)=r(0)+\int_0^t \cot (2r(s)) ds +\gamma(t),
\]
where $\gamma$ is a Brownian motion.
Let us then consider the local martingale:
\begin{align*}
D_t^\lambda =\exp \left( 2\lambda \int_0^t \cot (2r(s)) d\gamma(s) -2\lambda^2 \int_0^t \cot^2 (2r(s)) ds \right).
\end{align*}
From It\^o's formula, we compute
\[
D_t^\lambda =e^{2\lambda t}\left(\frac{\sin (2r(t))}{\sin (2r(0))}\right)^{\lambda} \exp\left(-2\lambda^2 \int_0^t \cot^2 (2r(s)) ds \right).
\]
Since $D_t\le e^{2\lambda t}(\sin (2r(0)))^{-\lambda}$, we know that $D$ is a martingale and as such, we may introduce the following probability measure $\mathbb{P}^\lambda$:
\[
\mathbb{P}_{/ \mathcal{F}_t} ^\lambda=D_t \mathbb{P}_{/ \mathcal{F}_t}= e^{2\lambda t}\left(\frac{\sin (2r(t))}{\sin (2r(0))}\right)^\lambda e^{- 2\lambda^2 \int_0^t \cot^2 (2r(s))ds} \mathbb{P}_{/ \mathcal{F}_t},
\]
where $\mathcal{F}$ is the natural filtration of $r$.
Denoting $\mathbb{E}^{\lambda}$ the corresponding expectation, we get:
\begin{equation}\label{integration1}
\mathbb{E}\left( e^{i \lambda \theta(t)} f(r(t)) \right)= e^{-2(\lambda^2+\lambda) t}\mathbb E^{\lambda}\left(\left(\frac{\sin (2r(0))}{\sin (2r(t))}\right)^{\lambda} f(r(t)) \right). 
\end{equation}
By Girsanov's theorem we know that the process
\[
\beta (t)=\gamma(t)-2\lambda\int_0^t \cot (2r(s))ds 
\]
is a Brownian motion under $\mathbb{P}^\lambda$.
As a matter of fact, 

\[
d r(t)= d\beta(t)+(2\lambda+1)\cot (2r(t))dt,
\] 
is a Jacobi diffusion under $\mathbb{P}^{\lambda}$ with generator
\[
\mathcal{L}^{\lambda,\lambda} :=\frac{1}{2} \frac{\partial^2}{\partial r^2}+\left(\left(\lambda+\frac{1}{2}\right)\cot (r)-\left(\lambda+\frac{1}{2}\right) \tan (r)\right)\frac{\partial}{\partial r}.
\]
From equation \eqref{integration1}, we therefore obtain
\begin{align*}
  & \int_0^{\pi/2} \mathbb{E}\left( e^{i \lambda \theta(t)} \mid r(t)=r \right)f(r)q_t^{0,0}(r(0),r) dr 
  = e^{-2(\lambda^2+\lambda) t}\int_0^{\pi/2} \left(\frac{\sin (2r(0))}{\sin (2r)}\right)^{\lambda} f(r)q_t^{\lambda,\lambda}(r(0),r) dr.
\end{align*}
Since this holds for every bounded Borel function $f$ the conclusion follows.
\end{proof}

It is not easy to invert the Fourier transform \eqref{fourier_conditional_radial_CP1}, however appealing to some of the methods in \cite{Nizar1} one can prove the following result:

\begin{proposition}
    Let $r \in (0,\pi/2)$.
    The conditional law $\mathbb{P}_{\theta (t) \mid r(t)=r}$ of $\theta (t)$ given $r(t)=r$ is given by
    \[
    d\mathbb{P}_{\theta(t) \mid r(t)=r}(\theta)= f_{r,t} (\theta) d\theta,
    \]
    where 
    \begin{align}\label{fr CP1}
    f_{r,t}(\theta)=\frac{1}{2\pi} \int_{\mathbb{R}}e^{-i\lambda \theta-2|\lambda|(|\lambda|+1)t}\left(\frac{\sin(2r(0))}{\sin(2r)}\right)^{|\lambda|}  \frac{q_t^{|\lambda|,|\lambda|}(r(0),r)}{q_t^{0,0}(r(0),r)} d\lambda
    \end{align}
    is a smooth function on $\mathbb R$ that satisfies for some $C(r,t)>0$,
    \begin{align}\label{CP1: fr}
    f_{r,t} (\theta) \sim_{|\theta| \to +\infty} \frac{C(r,t)}{\theta^2}.
    \end{align}
\end{proposition}

\begin{proof}
Formula \eqref{fr CP1} comes from the Fourier inversion formula applied to formula \eqref{fourier_conditional_radial_CP1}.
To study $f_{r,t}$ we will split 
\[
e^{-2|\lambda|(|\lambda|+1)t}\left(\frac{\sin(2r(0))}{\sin(2r)}\right)^{|\lambda|}  \frac{q_t^{|\lambda|,|\lambda|}(r(0),r)}{q_t^{0,0}(r(0),r)}
\]
into several parts, each of which we can compute the inverse Fourier transform of.
We first observe that 
\[
\lambda \to
e^{-2\lambda^2 t}\left(\sin(2r(0))\sin(2r)\right)^{|\lambda|}= e^{-2\lambda^2 t}e^{|\lambda|\ln [ \sin(2r(0)) \sin(2r) ] }
\]
is the Fourier transform of the convolution of a Gaussian distribution of variance $4t$ with a Cauchy distribution of parameter $-\ln [ \sin(2r(0)) \sin(2r) ]$.
Therefore, to study $f_{r,t}$, it is enough to understand the inverse Fourier transform of 
\[
\Psi(\lambda):=\frac{e^{-2|\lambda|t}}{2^{2|\lambda|}} \sum_{m=0}^{+\infty} (2m+2|\lambda|+1)\frac{\Gamma(m+2|\lambda|+1)\Gamma(m+1)}{\Gamma(m+|\lambda|+1)^2}  
  e^{-2m(m+2|\lambda|+1)t} P_m^{|\lambda|,|\lambda|}(\cos (2r(0)))P_m^{|\lambda|,|\lambda|}(\cos (2r)).
  \]
We now note that the Jacobi polynomial $P_m^{|\lambda|,|\lambda|}$ reduces (with different normalization) to the Gegenbauer polynomial $C_m^{|\lambda|+1/2}$, see equation \eqref{eq:Gegenbauer_polynomial_expression}.

Taking this into account, we obtain
 \[
\Psi(\lambda)= \frac{e^{-2|\lambda|t}}{2^{2|\lambda|}}\sum_{m=0}^{+\infty} (2m+2|\lambda|+1)\frac{\Gamma(m+2|\lambda|+1)}{\Gamma(|\lambda|+1)^2\Gamma(m+1)}  
  e^{-2m(m+2|\lambda|+1)t} \frac{C_m^{|\lambda|+1/2}(\cos (2r(0)))C_m^{|\lambda|+1/2}(\cos (2r))}{C_m^{|\lambda|+1/2}(1)^2}.
  \]
We now use the integral representation of Gegenbauer polynomials given by equation \eqref{representation Gegenbauer}
to see that $\Psi (\lambda)$ can be written as
\[
\Psi(\lambda) = \frac{1}{2^{2|\lambda|}}\frac{\Gamma(2|\lambda|+1)}{\Gamma(|\lambda|+1/2)\Gamma(|\lambda|+3/2)}\sum_{m=0}^{+\infty} Y_m(|\lambda| ) e^{-(2m(m+2|\lambda|+1)+2|\lambda|)t} \Phi^m_{r}(\lambda) \Phi^m_{r(0)}(\lambda),
\]
where $Y_m$ is a polynomial of degree $m+2$ and 
\begin{align*}
  \Phi^m_{r}(\lambda):=& \int_0^\pi \left( \cos (2r)+i \sin (2r) \cos (\eta) \right)^m (\sin (\eta ))^{2|\lambda|} d\eta \\
  =&\frac{1}{\pi}\int_{\mathbb{R}} e^{i\lambda \theta} \int_0^\pi \left( \cos (2r)+i \sin (2r) \cos (\eta )\right)^m \frac{-2\ln(\sin (\eta ))}{4\ln(\sin(\eta))^2+\theta^2}d\eta \, d\theta.
\end{align*}
The inverse Fourier transform of the term $Y_m(|\lambda| ) e^{-(2m(m+2|\lambda|+1)+2|\lambda|)t} $ is easy to compute since it is a linear combination of derivatives with respect to $t$ of $e^{-(4m+2)|\lambda|t}$, which is the Fourier transform of a Cauchy distribution.
We finally have from the duplication formula for the $\Gamma$-function (see \cite[Theorem 1.5.1]{special})
  \begin{align*}
  \frac{1}{2^{2|\lambda|}}\frac{\Gamma(2|\lambda|+1)}{\Gamma(|\lambda|+1/2)\Gamma(|\lambda|+3/2)}&=\frac{\Gamma(|\lambda|+1)}{\Gamma(1/2)\Gamma(|\lambda|+3/2)} \\
   &=\frac{1}{\pi} \int_0^1 \frac{s^{|\lambda|}}{\sqrt{1-s}}ds\\
   &=\frac{1}{\pi^2}\int_{\mathbb{R}} e^{i\lambda \theta}\int_0^1  \frac{-\ln(s)}{\ln(s)^2+\theta^2}\frac{ds}{\sqrt{1-s}} d\theta.
  \end{align*}
At this point, we inverted all the pieces making up $f_{r,t}$.
Writing an explicit formula for $f_{r,t}$ appears to be cumbersome, however it can be seen from our analysis that, when $|\theta| \to +\infty$ all the pieces decay at least as fast as $1/|\theta|^2$, hence the conclusion.
\end{proof}

We are now ready for the formula that gives the distribution of the index of the Brownian bridge.

\begin{theorem}\label{law winding CP1}
  
For $n \in \mathbb Z$, $t>0$, $w=\tan (r) e^{i \theta} \in \mathbb{C}P^1$, $r\in (0,\pi/2)$ and $\theta \in [0,2\pi)$ we have
\[
\mathbb{P}\left(  \theta(t) =\theta +2\pi n \mid w(t)= w\right)=2\pi \frac{\sin(2d(w(0),w))}{\sin(2d(0,w))} \frac{q_t^{0,0}(d(0,w(0)), d(0,w)))}{q_t^{0,0}(0,d(w(0),w))} \, f_{r,t}(\theta +2\pi n),
\]
and therefore  
\[
\mathbb{P}\left(  \theta(t) =\theta +2\pi n \mid w(t)= w\right) \sim_{|n| \to +\infty} \frac{C(w,t)}{n^2}
\]
for some $C(w,t) >0$. In particular, the law of the index of the Brownian loop of length $t$ is given by
\[
\mathbb{P}\left(  \theta(t) =2\pi n \mid w(t)= w(0) \right)=2\pi  \frac{\sum_{m=0}^{+\infty} (2m+1) e^{-2m(m+1)t} P_m^{0,0}(\cos (r(0)))^2}{ \sum_{m=0}^{+\infty} (2m+1) e^{-2m(m+1)t} } \, f_{r(0),t}(2\pi n).
\]
\end{theorem}

\begin{proof}
From Lemma 6.1 in \cite{Yor1980LoiDL} the conditional law of $\theta(t)$ given $w(t)=\tan (r) e^{i \theta}$ is given by 
\[
d\mathbb{P}_{\theta(t) |w(t)=\tan (r) e^{i \theta}} 
=\frac{U(\tan (r))}{\tan (r) \,  p_{\mathbb{C}P^1}(t,w(0), w)}\sum_{n\in\mathbb{Z}} f_{r,t}(\theta +2\pi n)\delta_{\theta +2\pi n} ,
\]
where we denote by $\delta_{\theta +2\pi n}$ the Dirac distribution at $\theta +2\pi n$ and by $U$ the density of $\tan(r(t))$ with respect to the Lebesgue measure.
The conclusion then follows from \eqref{heat kernel CPn} and a simple change of variable.
The asymptotics, when $|n|\to +\infty$, are a consequence of \eqref{CP1: fr}.
\end{proof}

\begin{corollary}\label{characteristic function bridge CP1}
For $\lambda \in \mathbb R$, $t>0$, $w=\tan (r) e^{i \theta} \in \mathbb{C}P^1$, $r\in (0,\pi/2)$ and $\theta \in [0,2\pi)$ we have

\begin{align*}
& \mathbb{E} \left( e^{i \lambda \theta (t)}  \mid w(t)=w\right) \\
=& \frac{\sin(2d(w(0),w))}{\sin(2r) q_t^{0,0}(0,d(w(0),w))} \sum_{n\in \mathbb{Z}}e^{-2|\lambda+n|(|\lambda+n|+1)t-in\theta}\left(\frac{\sin(2r(0))}{\sin(2r)}\right)^{|\lambda+n|}  q_t^{|\lambda+n|,|\lambda+n|}(r(0),r).
\end{align*}
In particular, for the characteristic function of the index of the Brownian loop of length $t$ one obtains
\[
\mathbb{E} \left( e^{i \lambda \theta (t)}  \mid w(t)=w(0)\right)= \frac{1}{\sin(2r) } \frac{\sum_{n\in \mathbb{Z}}e^{-2|\lambda+n|(|\lambda+n|+1)t}  q_t^{|\lambda+n|,|\lambda+n|}(r(0),r(0))}{ \sum_{m=0}^{+\infty} (2m+1) e^{-2m(m+1)t} }.
\]
\end{corollary}

\begin{proof}
Let $w=\tan (r) e^{i \theta} \in \mathbb{C}P^1$, $r\in (0,\pi/2)$, $\theta \in [0,2\pi)$.
From Theorem \ref{law winding CP1} we have 
    \begin{align*}
        \mathbb{E} \left( e^{i \lambda \theta (t)}  \mid w(t)=w\right)=2\pi \frac{\sin(2d(w(0),w))}{\sin(2r)} \frac{q_t^{0,0}(r(0), r))}{q_t^{0,0}(0,d(w(0),w))}\sum_{n \in \mathbb{Z}} f_{r,t}(\theta +2\pi n)e^{i\lambda (\theta+2\pi n)} .
    \end{align*}
 We now appeal to the Poisson summation formula \cite[Theorem 3.2.8]{Grafakos}, whose assumptions can be checked thanks to \eqref{CP1: fr}, to obtain
 \begin{align*}
 \sum_{n\in\mathbb{Z}} f_{r,t}( \theta +2\pi n)e^{2i\pi n \lambda }&=\sum_{n \in \mathbb Z} \int_{\mathbb R} f_{r,t}( \theta +2\pi \xi)e^{i(2\pi \lambda +2n\pi) \xi} d\xi \\
  &=\frac{1}{2\pi} \sum_{n \in \mathbb Z} \int_{\mathbb R}   f_{r,t}( u) e^{i(\lambda +n)(u-\theta)} du \\
  &=\frac{1}{2\pi} \sum_{n \in \mathbb Z}  \mathbb{E} \left( e^{i(\lambda +n)(\theta(t)-\theta)} \mid r(t)=r \right).
 \end{align*}
 The conclusion now easily follows from \eqref{fourier_conditional_radial_CP1} which yields
 \[
  \mathbb{E} \left( e^{i(\lambda +n)\theta(t)} \mid r(t)=r \right)=e^{-2|\lambda+n|(|\lambda+n|+1)t}\left(\frac{\sin(2r(0))}{\sin(2r)}\right)^{|\lambda+n|}  \frac{q_t^{|\lambda+n|,|\lambda+n|}(r(0),r)}{q_t^{0,0}(r(0),r)}
  \]
\end{proof}

\begin{proposition}
Let  $w=\tan (r) e^{i \theta} \in \mathbb{C}P^1$, $r\in (0,\pi/2)$,  and $\theta \in [0,2\pi)$.
Then, we have for every $\lambda \in \mathbb R$
\[
\lim_{t \to +\infty} \mathbb{E} \left( e^{i \lambda \frac{\theta (t)}{t}}  \mid w(t)=w\right)= e^{-2|\lambda|t}.
\]
Therefore, conditional on $w(t)=w$, $\frac{\theta(t)}{t}$ converges in distribution to a Cauchy distribution with parameter 2, when $t \to +\infty$.
\end{proposition}

\begin{proof}
From Corollary \ref{characteristic function bridge CP1} one can write
\begin{align*}
 & \mathbb{E} \left( e^{i \frac{\lambda}{t} \theta (t)}  \mid w(t)=w\right) \\
 =& \frac{\sin(2d(w(0),w))}{\sin(2r) q_t^{0,0}(0,d(w(0),w))} \sum_{n\in \mathbb{Z}}e^{-2|\frac{\lambda}{t}+n|(|\lambda+n|+1)t-in\theta}\left(\frac{\sin(2r(0))}{\sin(2r)}\right)^{|\frac{\lambda}{t}+n|}  q_t^{|\frac{\lambda}{t}+n|,|\frac{\lambda}{t}+n|}(r(0),r)
\end{align*}
and the result then follows using formula \eqref{jacobi heat kernel}.
\end{proof}

\subsection{Windings on \texorpdfstring{$\mathbb{S}^{2n+1}$}{S(2n+1)}}

We consider now the odd-dimensional sphere
\[
\mathbb{S}^{2n+1}=\left\{ z=(z_1,\cdots,z_{n+1}) \in \mathbb{C}^n\bigg| |z|^2:=\sum_{i=1}^n |z_i|^2=1 \right\}
\]
which is equipped with its canonical Riemannian structure. We parametrize the set  
\[
\mathbb{S}^{2n+1}\setminus \{ z \in \mathbb{S}^{2n+1} \mid z_{n+1} = 0 \}
\]
using the coordinates
\begin{align}\label{invarCP}
(w,\theta)\longrightarrow \frac{e^{i\theta} }{\sqrt{1+|w|^2}} \left( w,1\right),
\end{align}
where $\theta \in [0,2\pi)$ and $w \in \mathbb{C}^n$.
There is an isometric group action of the circle $\mathbb{S}^1=\mathbf{U}(1)$ on $\bS^{2n+1}$, which is defined by $$e^{i\theta}\cdot(z_1,\dots, z_n) = (e^{i\theta} z_1,\dots, e^{i\theta} z_n). $$

The quotient space $\bS^{2n+1} / \mathbf{U}(1)$ can be identified with $\mathbb{C}P^n$ and the projection map $$\pi : \bS^{2n+1} \to \mathbb{C}P^n$$ is a Riemannian submersion with totally geodesic fibers isometric to $\mathbf{U}(1)$.
The resulting fibration
\[
\mathbf{U}(1) \to \bS^{2n+1} \to \mathbb{C}P^n
\]
 is called the \emph{Hopf fibration}. From \eqref{invarCP}, it is clear that $\pi \left(\frac{e^{i\theta} }{\sqrt{1+|w|^2}} \left( w,1\right) \right)$ is the point in $\mathbb{C}P^n$ with local affine coordinates $w$ and that $\theta$ is the fiber coordinate in the sense that the vector field $\frac{\partial}{\partial \theta}$ is the generator of the $\mathbb{S}^1$ action on $\mathbb{S}^{2n+1}$.
 Therefore the coordinates $(w,\theta)$ are adapted to the geometry of the Hopf fibration.

 \begin{definition}
The one-form $d\theta$ on $\mathbb{S}^{2n+1}\setminus \{ z \in \mathbb{S}^{2n+1} \mid z_{n+1} = 0 \}$ is called the winding form around the great hypersphere $\{ z \in \mathbb{S}^{2n+1} \mid z_{n+1} = 0 \}$.
 \end{definition}

Note that $d\theta$ is closed but not exact and that the integral of  $d\theta$ along any loop is an integer multiple of $2\pi$.
From \eqref{invarCP}, one sees that $\mathbb{S}^{2n+1}\setminus \{ z \in \mathbb{S}^{2n+1} \mid z_{n+1} = 0 \}$ is diffeomorphic to $\mathbb{C}^n \times \mathbb{S}^1$.
Therefore the integral de Rham cohomology group $H_1^{dR}(\mathbb{S}^{2n+1}\setminus \{ z \in \mathbb{S}^{2n+1} \mid z_{n+1} = 0 \}, \mathbb{Z})$ is isomorphic to $\mathbb{Z}$ and $\frac{1}{2\pi} d\theta$ is one of the two generators of that cohomology. One of our goals in this section will to compute the exact  distribution of the law of the integral of $d\theta$ along Brownian loops. 
 
The Laplacian on $\mathbb{S}^{2n+1}$ will be denoted by $\Delta_{\mathbb{S}^{2n+1}}$ and for $\mu \ge 0$ we consider the \emph{Berger Laplacian}
 \[
\Delta_{\mathbb{S}^{2n+1}}^\mu:=\Delta_{\mathbb{S}^{2n+1}}+\left( \mu^2-1\right) \frac{\partial^2}{\partial \theta^2}.
 \]
For $\mu>0$, this operator is the Laplace-Beltrami operator of the so-called Berger metric on $\mathbb{S}^{2n+1}$ and for  $\mu=0$, $\Delta_{\mathbb{S}^{2n+1}}^\mu$ is the horizontal Laplacian of the Hopf fibration,  see \cite[Section 4.1.4]{Baudoin2022}.

Throughout this section, we fix $\mu \ge 0$ and we denote by $(\beta^\mu(t))_{t \ge 0}$ a diffusion process with generator $\frac{1}{2} \Delta_{\mathbb{S}^{2n+1}}^\mu$, which is started from the point with coordinates $\frac{e^{i\theta(0)} }{\sqrt{1+|w(0)|^2}} \left( w(0),1\right) $.
By rotational invariance, we can assume, without loss of generality, that $\theta(0)=0$.
Note that $(\beta^1(t))_{t \ge 0}$ is a Brownian motion and $(\beta^0(t))_{t \ge 0}$ a horizontal Brownian motion.
From \cite[Theorem 4.1.21]{Baudoin2022} we have
\[
\beta^\mu(t)= \frac{e^{i\theta(t)} }{\sqrt{1+|w(t)|^2}} \left( w(t),1\right)
\]
where:
\begin{enumerate}
\item $(w(t))_{t \ge 0}$ is a Brownian motion (in local affine coordinates) on $\mathbb{C}P^n$ started from $w(0)$;
\item $\theta(t)=\mu B(t) - \frac{i}{2}\sum_{j=1}^n \int_0^t \frac{w_j(s) d\overline{w}_j(s)-\overline{w}_j(s) dw_j(s)}{1+|w(s)|^2},
 $ with $(B(t))_{t \ge 0}$ one-dimensional Brownian motion started from 0, which is independent of $(w(t))_{t \ge 0}$.
\end{enumerate}
We will denote $r(t)=\arctan |w(t)|$. 
We recall that, as pointed out in Subsection \ref{BM CPn}, the process $(r(t))_{t\geq 0}$ is a Jacobi diffusion with generator $\mathcal{L}^{n-1,0}$.

\begin{proposition}\label{cor:characteristic_function_winding_number S2n+1}
For $\lambda \in \mathbb{R}$, $r\in \left[0,\frac{1}{2}\pi \right)$ and $t>0$
\begin{align*}
     \mathbb{E} (e^{i\lambda \theta(t)}\mid r(t)=r)=e^{-|\lambda| (n+\frac{1}{2}|\lambda|\mu^2 )t}\left(\frac{\cos (r(0))}{\cos (r)}\right)^{|\lambda| } \frac{q_t^{n-1,|\lambda |} (r(0),r)}{q_t^{n-1,0} (r(0),r)} .
\end{align*}
\end{proposition}

\begin{proof}
It follows, after adapting the proof of \cite[Theorem 4.1.17]{Baudoin2022} to the case $r(0)\neq 0$, that
\begin{align*}
    \mathbb{E} \left(e^{ \frac{\lambda}{2}\sum_{j=1}^n \int_0^t \frac{w_j(s) d\overline{w}_j(s)-\overline{w}_j(s) dw_j(s)}{1+|w(s)|^2}}\mid r(t)=r\right) =e^{-|\lambda|nt}\left(\frac{\cos (r(0))}{\cos (r)}\right)^{|\lambda| }\frac{q_t^{n-1,|\lambda |} (r(0),r)}{q_t^{n-1,0} (r(0),r)} .
\end{align*}
Since $\theta(t)=\mu B(t) - \frac{i}{2}\sum_{j=1}^n \int_0^t \frac{w_j(s) d\overline{w}_j(s)-\overline{w}_j(s) dw_j(s)}{1+|w(s)|^2},$ where $(B(t))_{t \ge 0}$ is a one-dimensional Brownian motion started from 0, which is independent of $(w(t))_{t \ge 0}$, the conclusion readily follows.
\end{proof}

\begin{proposition}
    Let $r \in [0,\pi/2)$.
    The conditional law $\mathbb{P}_{\theta (t) \mid r(t)=r}$ of $\theta (t)$ given $r(t)=r$ is given by
    \[
    d\mathbb{P}_{\theta(t) \mid r(t)=r}(\theta)= f_{r,t} (\theta) d\theta,
    \]
    where 
    \begin{align}\label{fr S2n+1}
    f_{r,t}(\theta) :=\frac{1}{2\pi} \int_{\mathbb{R}}e^{-i\lambda \theta-|\lambda| (n+\frac{1}{2}|\lambda|\mu^2 )t}\left(\frac{\cos(r(0))}{\cos(r)}\right)^{|\lambda|}  \frac{q_t^{n-1,|\lambda |} (r(0),r)}{q_t^{n-1,0} (r(0),r)} d\lambda
    \end{align}
    is a smooth function on $\mathbb R$ that satisfies
    \begin{align}\label{convo f to h}
    f_{r,t}(\theta)=\int_{-\infty}^{+\infty}\frac{e^{-\frac{(\theta-\eta)^2}{2\mu^2 t}}}{\sqrt{2\pi t} \mu} h_{r,t}(\eta) d\eta
    \end{align}
    with
    \begin{small}
    \begin{align*}
    h_{r,t}(\theta) := \frac{2 \sin (r)^{2n-1} \cos(r)}{\pi q_t^{n-1,0} (0,r)} \sum_{m=0}^{+\infty} e^{-2m(m+n)t} \left.Q_m\left(-\frac{1}{\phi_{m,r,r(0)}(t)} \frac{d}{ds}, \cos(2r), \cos(2r(0))\right) \right|_{s=1} \quad \frac{\phi_{m,r,r(0)}(t)s}{\phi_{m,r,r(0)}(t)^2s^2+\theta^2},
    \end{align*}
    \end{small}
    where $Q_m$ is a three-variable polynomial and $\phi_{m,r,r(0)}(t) :=(2m+n)t-\ln (\cos (r)\cos(r(0))$.
    In particular, for some $C(r,t)>0$,
    \begin{align}\label{S2n+1: fr}
    f_{r,t} (\theta) \sim_{|\theta| \to +\infty} \frac{C(r,t)}{\theta^2}.
    \end{align}
    
\end{proposition}
\begin{remark}
For $\mu=0$, we understand equation \eqref{convo f to h} as $f_{r,t}=h_{r,t}$.
    \end{remark}

\begin{proof}
The formula \eqref{fr S2n+1} simply follows from Proposition \ref{cor:characteristic_function_winding_number S2n+1} by the Fourier inversion formula. We then observe that
\begin{align*}
&  e^{-|\lambda| (n+\frac{1}{2}|\lambda|\mu^2 )t} \left(\frac{\cos(r(0))}{\cos (r)}\right)^{|\lambda|}\frac{q_t^{n-1,|\lambda |} (r(0),r)}{q_t^{n-1,0} (r(0),r)} \\
=&2 (\sin (r))^{2n-1} (\cos(r))^{|\lambda|+1} (\cos(r(0)))^{|\lambda|} \frac{e^{-|\lambda| (n+\frac{1}{2}|\lambda|\mu^2 )t}}{q_t^{n-1,0} (r(0),r)} \\
 & \sum_{m=0}^{+\infty} (2m+n+|\lambda|)\frac{\Gamma(m+n+|\lambda|)\Gamma(m+1)}{\Gamma(m+n)\Gamma(m+|\lambda|+1)}  
  e^{-2m(m+n+|\lambda|)t} P_m^{n-1,|\lambda|}(\cos(2r(0)))P_m^{n-1,|\lambda|}(\cos (2r)) \\
 =& 2 (\sin (r))^{2n-1} (\cos(r))^{|\lambda|+1} (\cos(r(0)))^{|\lambda|}\frac{e^{-|\lambda| (n+\frac{1}{2}|\lambda|\mu^2 )t}}{q_t^{n-1,0} (r(0),r)} \sum_{m=0}^{+\infty} e^{-2m(m+n+|\lambda|)t} Q_m( |\lambda|, \cos(2r),\cos(2r(0)))  \\
 =& \frac{2 (\sin (r))^{2n-1} \cos(r)}{q_t^{n-1,0} (r(0),r)} e^{-\frac{1}{2} \mu^2 \lambda^2 t}  \sum_{m=0}^{+\infty} e^{-2m(m+n)t} e^{-\phi_{m,r,r(0)}(t) |\lambda |} Q_m( |\lambda|, \cos(2r),\cos(2r(0))) \\
 =& \frac{2 (\sin (r))^{2n-1} \cos(r)}{q_t^{n-1,0} (r(0),r)} e^{-\frac{1}{2} \mu^2 \lambda^2 t} \sum_{m=0}^{+\infty} e^{-2m(m+n)t} \left.Q_m\left(-\frac{1}{\phi_{m,r,r(0)}(t)} \frac{d}{ds}, \cos(2r),\cos(2r(0))\right) \right|_{s=1}  e^{-\phi_{m,r,r(0)}(t)|\lambda |s} 
\end{align*}
where $Q_m$ is a three-variable polynomial which has degree $n+2m$ in the first variable. The asymptotics \eqref{S2n+1: fr} easily follows.
\end{proof}

\begin{theorem}\label{law winding S2n+1}
  
For $k \in \mathbb Z$, $t>0$, $\beta = e^{i \theta}\left(\frac{w}{\sqrt{1+|w|^2}},\frac{1}{\sqrt{1+|w|^2}}\right) \in \mathbb{S}^{2n+1}$, $w\in \mathbb{C}P^n$, $\theta \in [0,2\pi)$, and $r=\arctan(|w|)$ we have
\[
\mathbb{P}\left(  \theta(t) =\theta +2\pi k \mid \beta^\mu(t)= \beta\right)= \, \frac{f_{r,t}(\theta +2\pi k)}{\sum_{m \in \mathbb{Z}} f_{r,t}(\theta +2\pi m)},
\]
and therefore when $|k|\to +\infty$ 
\[
\mathbb{P}\left(  \theta(t) =\theta +2\pi k \mid \beta^\mu(t)= \beta\right) \sim_{|k| \to +\infty} \frac{C(\beta,t)}{k^2}
\]
for some $C(\beta,t) >0$. 

\end{theorem}

\begin{proof}
This follows immediately from Lemma 6.1 in \cite{Yor1980LoiDL} applied to the random variable $Z=\frac{e^{i\theta(t)}}{\sqrt{1+|w(t)|^2}}$ because
\[
\mathbb{P}\left(  \theta(t) =\theta +2\pi k \mid \beta^\mu(t)= \beta\right)=\mathbb{P}\left(  \theta(t) =\theta +2\pi k \mid Z= \frac{e^{i\theta}}{\sqrt{1+|w|^2}}\right) .
\]
The asymptotics when $n \to +\infty$ follow from Proposition \ref{cor:characteristic_function_winding_number S2n+1}.
\end{proof}

\begin{corollary}\label{characteristic function bridge S2n+1}
  Let  $\beta = e^{i \theta}\left(\frac{w}{\sqrt{1+|w|^2}},\frac{1}{\sqrt{1+|w|^2}}\right) \in \mathbb{S}^{2n+1}$, $w\in \mathbb{C}P^n$ and $\theta \in [0,2\pi)$.
  We have
  \[
  \mathbb{E} \left( e^{i \lambda \theta(t)}\mid \beta^\mu(t)= \beta \right)=\frac{\sum_{k \in \mathbb Z} e^{-ik\theta} e^{-|\lambda+k| (n+\frac{1}{2}|\lambda+k|\mu^2 )t}\frac{\cos (r(0))^{|\lambda+k| }}{\cos (r)^{|\lambda+k| }}q_t^{n-1,|\lambda+k|} (r(0),r)}{\sum_{k \in \mathbb Z} e^{-ik\theta} e^{-|k| (n+\frac{1}{2}|k|\mu^2 )t}\frac{\cos (r(0))^{|k| }}{\cos (r)^{|k| }}q_t^{n-1,|k |} (r(0),r)} ,
  \]
  for every $\lambda \in \mathbb{R}$.
\end{corollary}

\begin{proof}
The proof is similar to that of Corollary \ref{characteristic function bridge CP1}.  By appealing to the Poisson summation formula and Proposition \ref{cor:characteristic_function_winding_number S2n+1} we get
 \begin{align*}
 \sum_{k\in\mathbb{Z}} f_{r,t}( \theta +2\pi k)e^{2i\pi k \lambda }&=\frac{1}{2\pi} \sum_{k \in \mathbb Z}  \mathbb{E} \left( e^{i(\lambda +k)(\theta(t)-\theta)} \mid r(t)=r \right) \\
 &=\frac{1}{2\pi} \sum_{k \in \mathbb Z} e^{-i(\lambda +k)\theta} \mathbb{E} \left( e^{i(\lambda +k)\theta(t)} \mid r(t)=r \right) \\
 &=\frac{1}{2\pi} \sum_{k \in \mathbb Z} e^{-i(\lambda +k)\theta} e^{-|\lambda+k| (n+\frac{1}{2}|\lambda+k|\mu^2 )t}\frac{\cos (r(0))^{|\lambda+k| }}{\cos (r)^{|\lambda+k| }}\frac{q_t^{n-1,|\lambda+k |} (r(0),r)}{q_t^{n-1,0} (r(0),r)},
 \end{align*}
 and the conclusion follows.
\end{proof}

\begin{corollary}
  Let  $\beta = e^{i \theta}\left(\frac{w}{\sqrt{1+|w|^2}},\frac{1}{\sqrt{1+|w|^2}}\right) \in \mathbb{S}^{2n+1}$, $w\in \mathbb{C}P^n$ and $\theta \in [0,2\pi)$.
  We have
  \[
 \lim_{t\to +\infty} \mathbb{E} \left( e^{i \lambda \frac{\theta(t)}{t}}\mid \beta^\mu(t)= \beta \right)=e^{-n |\lambda |t}.
  \]
  Therefore, conditionally to $\beta^\mu(t)= \beta$, when $t\to +\infty$, one has  
\[
\frac{\theta(t)}{t} \to C_n
\]
where $C_n$ is a Cauchy distribution with parameter $n$.
\end{corollary}
\begin{proof}
The proof is a simple application of Corollary \ref{characteristic function bridge S2n+1}.
\end{proof}

\section{Windings  on the anti-de Sitter fibration}

We now turn to the second part of the paper which concerns windings related to the anti-de Sitter fibration.

\subsection{Preliminaries}

\subsubsection{Hyperbolic Jacobi diffusions}

A hyperbolic Jacobi process is a diffusion on $[0,+\infty)$ with generator:
\begin{align*}
    \mathcal{L}^{\alpha ,\beta } :=\frac{1}{2}\frac{\partial^2}{\partial r^2}+\left(\left(\alpha +\frac{1}{2}\right)\coth (r) +\left(\beta +\frac{1}{2}\right) \tanh (r)\right)\frac{\partial}{\partial r} 
\end{align*}
for $\alpha >-1$ and $\beta\in\mathbb{R}$.
As follows from \cite[Theorem 2.3]{Koornwinder1984}, the transition density $q_t^{\alpha,\beta}(r(0),\cdot)$ with respect to the Lebesgue measure of a diffusion with hyperbolic Jacobi operator $\mathcal{L}^{\alpha ,\beta }$ is given by
    \begin{align}\label{jacobi hyperbolic heat kernel}
        q_t^{\alpha,\beta} (r(0),r)=\frac{(\sinh (r))^{2\alpha +1}(\cosh (r))^{2\beta +1}}{\pi }\int_0^{\infty}e^{-\frac{1}{2}(p^2+(\alpha +\beta +1)^2) t}\Phi_p^{\alpha ,\beta } (r(0)) \Phi_p^{\alpha ,\beta }(r) m_{\alpha ,\beta }(p) dp ,
    \end{align}
    where
    \begin{align*}
        \Phi_p^{\alpha ,\beta}(r) := _2 F_1 \left( \frac{\alpha +\beta +1+ip}{2} , \frac{\alpha +\beta +1-ip}{2}, \alpha +1; -\sinh^2(r) \right) ,
        \end{align*}
        
        \begin{align}\label{harish-chandra function}
        m_{\alpha ,\beta } (p):=&\left|\frac{\Gamma\left(\frac{\alpha +\beta +1+ip}{2}\right)\Gamma\left(\frac{\alpha -\beta +1+ip}{2}\right)}{\Gamma (1+\alpha )\Gamma (ip)}\right|^2 
    \end{align}
    and $_2 F_1$ is the Gauss hypergeometric function.
    If $p$ is non negative, the spherical function $\phi^{\alpha,\beta}_p(r)$ is the unique solution of
\[
\mathcal{L}^{\alpha,\beta}f=-\frac{1}{2}(p^2+(\alpha+\beta+1)^2) f
\]
such that $\phi^{\alpha,\beta}_p(0)=1$ and $(\phi_p^{\alpha,\beta})'(0)=0$.

\subsubsection{Brownian motion on \texorpdfstring{$\mathbb{C}H^n$}{CHn}}\label{prelim BM CHn}

As a differential manifold, the complex hyperbolic space  $\mathbb{C}H^n$ can be defined as the open unit ball in $\mathbb{C}^n$:
 \[
\mathbb{C}H^n=\{ z \in \mathbb{C}^{n} | \, | z_1|^2+\cdots+|z_n|^2 <1 \}.
 \]
The  Riemannian structure of  $\mathbb{C}H^n$ can be constructed as follows.
Let 
\[
\mathbf{AdS}^{2n+1}=\{ z \in \mathbb{C}^{n+1} | \, | z_1|^2+\cdots+|z_n|^2 -|z_{n+1}|^2=-1 \}
\]
be the $2n+1$ dimensional anti-de Sitter space. We equip $\mathbf{AdS}^{2n+1}$ with its standard Lorentz metric with signature $(2n,1)$ inherited from $\mathbb C^{n+1}$.
The Riemannian structure on $\mathbb{C}H^n$ is then such that the map
\begin{align*}
\begin{array}{llll}
\pi :& \mathbf{AdS}^{2n+1} & \to & \mathbb{C}H^n \\
 & (z_1,\dots,z_{n+1}) & \mapsto & \left( \frac{z_1}{z_{n+1}}, \dots, \frac{z_n}{z_{n+1}}\right)
\end{array}
\end{align*}
is an indefinite Riemannian submersion whose 1-dimensional fibers are negative definite.
This submersion is associated with a fibration.
Indeed, the group $\mathbf{U}(1)$ acts isometrically on $\mathbf{AdS}^{2n+1}$, and the quotient space of $\mathbf{AdS}^{2n+1}$ by this action is isometric to $\mathbb{C}H^n$.
The resulting fibration
\[
\mathbf{U}(1)\to \mathbf{AdS}^{2n+1}\to\mathbb{C}H^n
\]
 is called the complex anti-de Sitter fibration and the Riemannian metric on $\mathbb{C}H^n$ the Bergman metric. We refer to \cite[Section 4.2]{Baudoin2022} for further details.

To parametrize $\mathbb{C}H^n$, we will use the global affine coordinates given by $w_j=z_j/z_{n+1}$ where $z\in \{z\in \mathbb{C}^{n+1} | \sum_{k=1}^n|z_{k}|^2-|z_{n+1}|^2<0 \}$.
Those coordinates provide a diffeomorphism between $\mathbb{C}H^n$ and the open unit ball in $\mathbb C^n$.
In affine coordinates, the Laplace-Beltrami operator for the Bergman metric of $\mathbb{C}H^n$ is given by
\begin{align}\label{Laplacian CHn}
\Delta_{\mathbb{C}H^n}=4(1-|w|^2)\sum_{k=1}^n \frac{\partial^2}{\partial w_k \partial\overline{w}_k}- 4(1-|w|^2)\mathcal{R} \overline{\mathcal{R}}
\end{align}
where $|w|^2:=\sum_{i=1}^n |w_i|^2 <1$ and
\[
\mathcal{R}:=\sum_{j=1}^n w_j \frac{\partial}{\partial w_j} .
\] 
The Bergman metric $h$ on $\mathbb{C}H^n$ satisfies
\[
h \left(\frac{\partial}{\partial w_i}, \frac{\partial}{\partial \overline{w}_j} \right)= \frac{(1-|w|^2) \delta_{ij} +\overline{w}_i w_j}{(1-|w|^2)^2}.
\]
The Riemannian volume measure is given 
\[
\frac{dw}{(1-|w|^2)^{n+1}} ,
\]
where $dw$ denotes the Lebesgue measure on the open unit ball. 
The Riemannian distance on $\mathbb{C}H^n$ between two points of affine coordinates $w$ and $\tilde{w}$ will be denoted by $d(w,\tilde{w})$.
One can check that if a point has coordinate $w$, then
\[
d(0,w)=\arctanh (|w|).
\]

Brownian motion on $\mathbb{C}H^n$ is a diffusion with generator $\frac{1}{2}\Delta_{\mathbb{C}H^n}$.
The radial part of a Brownian motion on $\mathbb{C}H^n$ is a Jacobi diffusion.
More precisely, let us consider a Brownian motion $(w(t))_{t \ge 0}$ on $\mathbb{C}H^n$ issued from a point $w(0) $.
The process $r(t):=d(0,w(t))=\arctanh (|w(t)|)$ is a hyperbolic Jacobi diffusion with generator $\mathcal{L}^{n-1,0}$ from which it follows that  the density of $w(t)$, with respect to the Lebesgue measure, is given by the formula
\begin{align}\label{heat kernel CHn}
p_{\mathbb{C}P^n}(t,w(0), w)= \frac{\Gamma \left(n \right)}{2 \pi^{n}}  \frac{q_t^{n-1,0}(0,d(w(0),w))}{\cosh(d(w(0),w)) \sinh(d(w(0),w))^{2n-1} } \frac{1}{(1-|w|^2)^{n+1}}. 
\end{align}
at a point $w$ and a time $t>0$.
When $w=w(0)$ we of course understand formula \eqref{heat kernel CHn} with the help of formula \eqref{jacobi hyperbolic heat kernel}.

\subsection{Windings on \texorpdfstring{$\mathbb{C}H^1$}{CH1}}

In this section we will look at the winding number of the Brownian motion $(w(t))_{t \ge 0}$ on $\mathbb{C}H^1$ around zero.
We will assume that $(w(t))_{t \ge 0}$ is started from $w(0) \neq 0$.
The point $0$ is polar for the process $(w(t))_{t\ge 0}$ one can therefore write a polar decomposition
\[
w(t) =| w(t)| e^{i\theta(t)} ,
\]
where 
\[
\theta(t)=\theta(0)+\frac{1}{2i}\int_0^t \frac{ \overline{w}(s)dw(s)-d\overline{w}(s)w(s)}{|w(s)|^2} ,
\]
and $\theta(0)$ is such that $w(0)=|w(0)|e^{i\theta(0)}$.
Using rotational invariance we can assume $\theta(0)=0$.

Consequently, by the above and since we are working in conformal coordinates, it is natural to make the following definition:
\begin{definition}\index{winding!processCH1}
The winding process of the Brownian motion $(w(t))_{t\geq 0}$ on $\mathbb{C}H^1$ is defined by 
\begin{equation*}
\theta(t) := \frac{1}{2i}\int_0^t \frac{ \overline{w}(s)dw(s)-d\overline{w}(s)w(s)}{|w(s)|^2}\textrm{ for } t \geq 0 .
\end{equation*}
\end{definition}

From \eqref{Laplacian CHn} the generator of Brownian motion in $\mathbb{C}H^1$, in polar coordinates $(r,\theta )$ given by $w=\tanh (r)e^{i\theta}$, reads:
\[
\frac{1}{2}\frac{\partial^2}{\partial r^2} + \coth (2r) \frac{\partial}{\partial r} +\frac{2}{\sinh^2 (2r)} \frac{\partial^2}{\partial\theta^2}\textrm{ for } r\in\mathbb{R}\textrm{ and } \theta\in\mathbb{R},
\]
where $r$ parametrises the Riemannian distance from $0$ in $\mathbb{C}H^1$.
This shows that the winding process $(\theta(t))_{t\geq 0}$ of Brownian motion in $\mathbb{C}H^1$ is given by
\begin{align}\label{eq:clockWindingCH1}
\theta(t) =\beta\left( \int_0^t \frac{4ds}{\sinh^2 (2r_s)}\right)\textrm{ for } t \geq 0,
\end{align}
where $(r(t))_{t\geq 0}$ is the Jacobi diffusion started at $r(0)\in (0,\pi /2)$ with generator 
\[
\frac{1}{2} \frac{\partial^2}{\partial r^2} + \coth (2r) \frac{\partial}{\partial r}
\]
and $(\beta (t))_{t\geq 0}$ is a Brownian motion independent from $r$.

\begin{proposition}\label{prop:Fourier(t)ransform_CH}
For $t >0$ and $r>0$, the conditional Fourier transform of $(\theta(t))_{t\geq 0}$ given $r(t)$ is given by
\begin{align*}
\mathbb{E} \left( e^{i \lambda \theta(t)}  \mid r(t)=r\right)
& = \left(\frac{\tanh (r(0))}{\tanh (r)}\right)^{|\lambda |} \frac{q_t^{|\lambda |,-|\lambda |}(r(0),r)}{q_t^{0,0}(r(0),r)} .
\end{align*}
\end{proposition}
\begin{proof}
Using symmetry, we can assume $\lambda >0$.
From equation \eqref{eq:clockWindingCH1} we have that for every bounded Borel function $f$
\begin{align}\label{eq:Laplace-transform-complex-hyperbolic-line}
\begin{split}
    \mathbb{E}\left( e^{i\lambda\theta(t)} f(r(t))\right) =&\mathbb{E}\left(\exp\left(-\lambda\beta_{\int_0^t \frac{4ds}{\sinh^2 (2r_s)}}\right) f(r(t))\right) 
    =\mathbb{E}\left(\exp\left(-2\lambda^2\int_0^t\frac{ds}{\sinh^2 (2r_s)}\right) f(r(t))\right)\\
    =&e^{2\lambda^2 t}\mathbb{E}\left(\exp\left(-2\lambda^2\int_0^t \coth^2 (2r_s)ds\right) f(r(t))\right) ,
\end{split}
\end{align}
where $(\beta (t))_{t\geq 0}$ is a Brownian motion independent from $(r(t))_{t\geq 0}$.
Observe that $(r(t))_{t\geq 0}$ is the unique strong solution to the stochastic differential equation
\begin{align*}
    r(t) =r(0) +\int_0^t\coth (2r(s))ds+\gamma (t)
\end{align*}
for some Brownian motion $(\gamma (t))_{t\geq 0}$.

To obtain the expression of the Fourier transform we consider the local martingale:
\begin{align*}
M^\lambda (t) =\exp\left( 2\lambda \int_0^t\frac{d\gamma (s)}{\sinh (2r(s))}  -2\lambda^2\int_0^t\frac{ds}{\sinh^2 (2r(s))} \right) .
\end{align*}
From It\^o's formula, we compute
\[
M^\lambda (t)=\left(\frac{\tanh (2r(t))}{\tanh (2r(0))}\right)^{\lambda}\exp\left(-2\lambda^2 \int_0^t\frac{ds}{\sinh^2 (2r(s))} \right) .
\]
Since $(M^{\lambda}(t))_{t\geq 0}$ is bounded from above by $(\tanh (2r(0)))^{-\lambda}$ it is a martingale and as such, we may introduce the following probability measure $\mathbb{P}^\lambda$:
\[
\mathbb{P}_{| \mathcal{F}_t} ^\lambda
=M^{\lambda}(t) \mathbb{P}_{| \mathcal{F}_t}
=\left(\frac{\tanh (2r(t))}{\tanh (2r(0))}\right)^{\lambda}\exp\left(- 2\lambda^2 \int_0^t\frac{1}{\sinh^2 (2r(s))}ds\right) \mathbb{P}_{| \mathcal{F}_t},
\]
where $(\mathcal{F}_t)_{t\geq 0}$ is the natural filtration of $(r(t))_{t\geq 0}$.
Letting $\mathbb{E}^{\lambda}$ denote the corresponding expectation, we get
\begin{equation}\label{eq:integration1CH1}
\mathbb{E}\left( e^{i \lambda \theta(t)} f(r(t)) \right)= \mathbb{E}^{\lambda}\left(\left(\frac{\tanh (2r(0))}{\tanh (2r(t))}\right)^{\lambda} f(r(t))\right) .
\end{equation}
By Girsanov's theorem we know that the process
\[
\beta (t)=\gamma (t)-2\lambda\int_0^t\frac{ds}{\sinh (2r (s))} =\gamma (t)-2\lambda\int_0^t\coth (2r (s))ds +2\lambda\int_0^t\tanh (r (s))ds
\]
is a Brownian motion under $\mathbb{P}^\lambda$.
As a matter of fact 
\[
r(t)= \beta (t) +(2\lambda+1)\int_0^t\coth (2r (s))ds -2\lambda\int_0^t\tanh (r(s))ds,
\] 
is a hyperbolic Jacobi diffusion under $\mathbb{P}^{\lambda}$ with generator
\[
\mathcal{L}^{\lambda,\lambda}=\frac{1}{2} \frac{\partial^2}{\partial r^2}+\left(\left(\lambda+\frac{1}{2}\right)\coth (r)+\left( -\lambda+\frac{1}{2}\right)\tanh (r)\right)\frac{\partial}{\partial r}.
\]
From equation \eqref{eq:integration1CH1} we therefore obtain
\begin{align*}
   \int_0^{\infty} \mathbb{E}\left( e^{i \lambda \theta(t)} \mid r(t)=r \right)f(r)q_t^{0,0}(r(0),r) dr 
 = \int_0^{\infty}\left(\frac{\tanh (2r(0))}{\tanh (2r(t))}\right)^{\lambda} f(r) q_t^{\lambda,-\lambda}(r(0),r) dr.
\end{align*}
Since this holds for every bounded Borel function $f$ we obtain the result.

\end{proof}

\begin{proposition}
    Let $r  >0$.  The conditional law $\mathbb{P}_{\theta (t) \mid r(t)=r}$ of $\theta (t)$ given $r(t)=r$ is given by
    \[
    d\mathbb{P}_{\theta(t) \mid r(t)=r}(\theta)= f_{r,t} (\theta) d\theta,
    \]
    where 
    \begin{align}\label{fr CH1}
    f_{r,t}(\theta) = \frac{\sinh(2r) }{2\pi^2 q_t^{0,0}(r(0),r)}  \int_0^{\infty} \int_0^1 \int_0^1 &  e^{-\frac{1}{2}(p^2+1) t} \frac{v^{-\frac{1+ip}{2}}(1-v)^{- \frac{1-ip}{2}} u^{-\frac{1-ip}{2}}(1-u)^{ - \frac{1+ip}{2}}  }{\left( 1+v\sinh(r)^2 \right)^\frac{1+ip}{2}\left( 1+u\sinh(r(0))^2 \right)^\frac{1-ip}{2}}  \\
     &  \frac{-\ln \left[ \tanh (r(0))\tanh (r)(1-u)(1-v)\right] }{ (\ln \left[ \tanh (r(0))\tanh (r)(1-u)(1-v)\right] )^2+\theta^2}  \left|\frac{\Gamma\left(\frac{1+ip}{2}\right)}{\Gamma (ip)\Gamma \left(\frac{1-ip}{2}\right)}\right|^2 \, du dv dp. \notag
    \end{align}
    is a smooth function on $\mathbb R$ that satisfies for some $C(r,t)>0$,
    \begin{align}\label{CH1: fr}
    f_{r,t} (\theta) \sim_{|\theta| \to +\infty} \frac{C(r,t)}{\theta^2}.
    \end{align}
\end{proposition}

\begin{proof}
From equation \eqref{jacobi hyperbolic heat kernel} and Proposition \ref{prop:Fourier(t)ransform_CH} we have
\begin{align*}
& \mathbb{E} \left( e^{i \lambda \theta(t)}  \mid r(t)=r\right) \\
 =& \left(\frac{\tanh (r(0))}{\tanh (r)}\right)^{|\lambda |} \frac{q_t^{|\lambda |,-|\lambda |}(r(0),r)}{q_t^{0,0}(r(0),r)} \\
 =&\sinh(2r) \frac{\left(\tanh (r(0))\tanh (r)\right)^{|\lambda |} }{2\pi q_t^{0,0}(r(0),r)}\int_0^{\infty}e^{-\frac{1}{2}(p^2+1) t}\Phi_p^{|\lambda |,-|\lambda |} (r(0)) \Phi_p^{|\lambda |,-|\lambda |}(r) m_{|\lambda| ,-|\lambda| }(p) dp.
\end{align*}
Since $ \Phi_p^{|\lambda |,-|\lambda |}$ is an hypergeometric function, it admits the following Euler integral representation, see \cite[Section 2.2]{special}:
\[
\Phi_p^{|\lambda |,-|\lambda |}(r)=\frac{\Gamma(|\lambda|+1)}{\Gamma \left(\frac{1-ip}{2}\right)\Gamma \left(|\lambda|+\frac{1+ip}{2}\right)} \int_0^1 v^{-\frac{1+ip}{2}}(1-v)^{|\lambda| - \frac{1-ip}{2}} \frac{dv}{\left( 1+v\sinh(r)^2 \right)^\frac{1+ip}{2}}.
\]
From \eqref{harish-chandra function} we have
 \begin{align*}
        m_{|\lambda| ,-|\lambda| }(p)=&\left|\frac{\Gamma\left(\frac{1+ip}{2}\right)\Gamma\left(\frac{2|\lambda| +1+ip}{2}\right)}{\Gamma (1+|\lambda| )\Gamma (ip)}\right|^2. 
    \end{align*}
Therefore, using the fact that $\Phi_p^{|\lambda |,-|\lambda |} (r(0)) =\overline{\Phi_p^{|\lambda |,-|\lambda |} (r(0)) }$ we obtain
\begin{align*}
 \mathbb{E} \left( e^{i \lambda \theta(t)}  \mid r(t)=r\right) & =\sinh(2r) \frac{\left(\tanh (r(0))\tanh (r)\right)^{|\lambda |} }{2\pi q_t^{0,0}(r(0),r)}  \\ 
  &\int_0^{\infty} \int_0^1 \int_0^1 e^{-\frac{1}{2}(p^2+1) t} \frac{v^{-\frac{1+ip}{2}}(1-v)^{|\lambda| - \frac{1-ip}{2}} u^{-\frac{1-ip}{2}}(1-u)^{|\lambda| - \frac{1+ip}{2}}  }{\left( 1+v\sinh(r)^2 \right)^\frac{1+ip}{2}\left( 1+u\sinh(r(0))^2 \right)^\frac{1-ip}{2}} \left|\frac{\Gamma\left(\frac{1+ip}{2}\right)}{\Gamma (ip)\Gamma \left(\frac{1-ip}{2}\right)}\right|^2du dv dp.
\end{align*}   
Using the formula for the Fourier transform of a Cauchy distribution, we now write
\begin{align*}
\left(\tanh (r(0))\tanh (r)\right)^{|\lambda |}(1-u)^{|\lambda|} (1-v)^{|\lambda|}&=e^{-\left(-\ln \left[ \tanh (r(0))\tanh (r)(1-u)(1-v)\right] \right) |\lambda|} \\
 &=\frac{1}{\pi} \int_0^{+\infty} e^{i\lambda \theta} \frac{-\ln \left[ \tanh (r(0))\tanh (r)(1-u)(1-v)\right] }{ (\ln \left[ \tanh (r(0))\tanh (r)(1-u)(1-v)\right] )^2+\theta^2} d\theta
\end{align*} 
and the stated result follows.
\end{proof}

\begin{theorem}\label{law winding CH1}
  
For $n \in \mathbb Z$, $t>0$, $w=\tanh (r) e^{i \theta} \in \mathbb{C}H^1$, $r\in (0,+\infty)$, and $\theta \in [0,2\pi)$ we have
\[
\mathbb{P}\left(  \theta(t) =\theta +2\pi n \mid w(t)= w\right)=2\pi \frac{\sinh(2d(w(0),w))}{\sinh(2d(0,w))} \frac{q_t^{0,0}(d(0,w(0)), d(0,w)))}{q_t^{0,0}(0,d(w(0),w))} \, f_{r,t}(\theta +2\pi n),
\]
and therefore  
\[
\mathbb{P}\left(  \theta(t) =\theta +2\pi n \mid w(t)= w\right) \sim_{|n| \to +\infty} \frac{C(w,t)}{n^2}
\]
for some $C(w,t) >0$. 
\end{theorem}

\begin{proof}
The proof is similar to that of Theorem \ref{law winding CP1}.
\end{proof}

We are next interested in asymptotics when $t \to +\infty$. We will use  the following lemma.

\begin{lemma}\label{lemma:asymptotics_general_hyperbolic_heat_kernel}
   For $\alpha>-1$ and $\beta \in\mathbb{R}$,  we have the following asymptotics
   \begin{align*}
       q_t^{\alpha,\beta}(r(0),r) \underset{t \to +\infty}{\sim} \frac{(\sinh(r))^{2\alpha +1}(\cosh (r))^{2\beta +1}}{\sqrt{2\pi} t^{3/2}} e^{-\frac{1}{2} (\alpha+\beta+1)^2t} \Phi_0^{\alpha,\beta}(r(0)) \Phi_0^{\alpha,\beta}(r)\left|\frac{\Gamma\left(\frac{\alpha +\beta +1}{2}\right)\Gamma\left(\frac{\alpha -\beta +1}{2}\right)}{\Gamma (1+\alpha )}\right|^2,
   \end{align*}
   when $t \to +\infty$
    
\end{lemma}
\begin{proof}

The subsitution $p=u/\sqrt{t}$ gives
\begin{align*}
    \int_0^{\infty}e^{-\frac{1}{2}p^2t}\Phi_p^{\alpha ,\beta } (r(0)) \Phi_p^{\alpha ,\beta }(r) m_{\alpha ,\beta }(p) dp=\frac{1}{\sqrt{t}}\int_0^{\infty}e^{-\frac{1}{2}u^2}\Phi_{u/\sqrt{t}}^{\alpha ,\beta }(r(0))\Phi_{u/\sqrt{t}}^{\alpha ,\beta } (r)m_{\alpha ,\beta }\left(\frac{u}{\sqrt{t}}\right) du,
\end{align*}
where
\begin{align*}
   m_{\alpha,\beta }\left(\frac{u}{\sqrt{t}}\right)&=\left|\frac{\Gamma\left(\frac{\alpha +\beta +1+iu/\sqrt{t}}{2}\right)\Gamma\left(\frac{\alpha -\beta +1+iu/\sqrt{t}}{2}\right)}{\Gamma (1+\alpha )\Gamma (iu/\sqrt{t})}\right|^2  \\
   & \underset{t \to +\infty}{\sim}   \frac{u^2}{t} \left|\frac{\Gamma\left(\frac{\alpha +\beta +1}{2}\right)\Gamma\left(\frac{\alpha -\beta +1}{2}\right)}{\Gamma (1+\alpha )}\right|^2
 \end{align*}
 Therefore, thanks to (6.3) and (6.4) in  \cite{Koornwinder1984} which allow to use dominated convergence,  we obtain
 \begin{align*}
 \int_0^{\infty}e^{-\frac{1}{2}p^2t}\Phi_p^{\alpha ,\beta } (r(0)) \Phi_p^{\alpha ,\beta }(r) m_{\alpha ,\beta }(p) dp\underset{t \to +\infty}{\sim} \frac{\pi^{1/2}}{\sqrt{2}t^{3/2}}\Phi_0^{\alpha,\beta}(r(0)) \Phi_0^{\alpha,\beta}(r)\left|\frac{\Gamma\left(\frac{\alpha +\beta +1}{2}\right)\Gamma\left(\frac{\alpha -\beta +1}{2}\right)}{\Gamma (1+\alpha )}\right|^2.
 \end{align*}
    The conclusion then follows from \eqref{jacobi hyperbolic heat kernel}.
\end{proof}
The next theorem gives the asymptotic law when $t\to +\infty$ of the index of the Brownian loop of length $t$.

 \begin{theorem}
For $n \in \mathbb Z$, $w=\tanh (r) e^{i \theta} \in \mathbb{C}H^1$, $r\in (0,+\infty)$, and $\theta \in [0,2\pi)$ we have
\begin{small}
\begin{align*}
  \lim_{ t \to +\infty} \mathbb{P} &\left(  \theta(t)  =\theta +2\pi n \mid w(t)  = w\right)   
 = \frac{2}{\pi^2 \Phi_0^{0,0}(d(w(0),w)) }\cdot \\
 &\int_0^1 \int_0^1 \frac{v^{-\frac{1}{2}}(1-v)^{- \frac{1}{2}} u^{-\frac{1}{2}}(1-u)^{ - \frac{1}{2}}  }{\left( 1+v\sinh(r)^2 \right)^\frac{1}{2}\left( 1+u\sinh(r(0))^2 \right)^\frac{1}{2}} 
\frac{-\ln \left[ \tanh (r(0))\tanh (r)(1-u)(1-v)\right] }{ (\ln \left[ \tanh (r(0))\tanh (r)(1-u)(1-v)\right] )^2+(\theta+2\pi n)^2}   \, du dv . 
    \end{align*}
    \end{small}
\end{theorem}

\begin{proof}
To compute the limit, we directly use the formula in Theorem \ref{law winding CH1}.
We first get, from Lemma \ref{lemma:asymptotics_general_hyperbolic_heat_kernel}, that
\[
q_t^{0,0}(r(0),r) \underset{t \to +\infty}{\sim} e^{-\frac{1}{2}t} \frac{\pi^{3/2} \sinh(2r)}{2\sqrt{2} t^{3/2}} \Phi_0^{0,0}(r(0)) \Phi_0^{0,0}(r).
\]
This yields
\[
\lim_{t \to +\infty} \frac{q_t^{0,0}(d(0,w(0)), d(0,w)))}{q_t^{0,0}(0,d(w(0),w))} =\frac{\sinh(2d(0,w))}{\sinh(2d(w(0),w))} \frac{\Phi_0^{0,0}(r(0)) \Phi_0^{0,0}(r)}{ \Phi_0^{0,0}(d(w(0),w))}.
\]
It remains therefore to compute $\lim_{t\to +\infty} f_{r,t}(\theta)$.
In the integral \eqref{fr CH1} we perform the change of variable $s=p \sqrt{t}$ which yields
\begin{align*}
    f_{r,t}(\theta) = \frac{e^{-\frac{1}{2}t} \sinh(2r) }{2\pi^2\sqrt{t} q_t^{0,0}(r(0),r)} & \int_0^{\infty} \int_0^1 \int_0^1   e^{-\frac{1}{2}s^2} \frac{v^{-\frac{1+is/\sqrt{t}}{2}}(1-v)^{- \frac{1-is/\sqrt{t}}{2}} u^{-\frac{1-is/\sqrt{t}}{2}}(1-u)^{ - \frac{1+is/\sqrt{t}}{2}}  }{\left( 1+v\sinh(r)^2 \right)^\frac{1+is/\sqrt{t}}{2}\left( 1+u\sinh(r(0))^2 \right)^\frac{1-is/\sqrt{t}}{2}}\cdot  \\
      &  \frac{-\ln \left[ \tanh (r(0))\tanh (r)(1-u)(1-v)\right] }{ (\ln \left[ \tanh (r(0))\tanh (r)(1-u)(1-v)\right] )^2+\theta^2}  \left|\frac{\Gamma\left(\frac{1+is/\sqrt{t}}{2}\right)}{\Gamma (is/\sqrt{t})\Gamma \left(\frac{1-is/\sqrt{t}}{2}\right)}\right|^2 \, du dv ds .
    \end{align*}
    Using similar methods as in the proof of Lemma \ref{lemma:asymptotics_general_hyperbolic_heat_kernel}, we obtain
   \begin{align*}
    f_{r,t}(\theta) \underset{t\to +\infty}{\sim} \frac{1}{\pi^3 \Phi_0^{0,0}(r(0)) \Phi_0^{0,0}(r)}   \int_0^1 \int_0^1 &  \frac{v^{-\frac{1}{2}}(1-v)^{- \frac{1}{2}} u^{-\frac{1}{2}}(1-u)^{ - \frac{1}{2}}  }{\left( 1+v\sinh(r)^2 \right)^\frac{1}{2}\left( 1+u\sinh(r(0))^2 \right)^\frac{1}{2}}\cdot \\
      &  \frac{-\ln \left[ \tanh (r(0))\tanh (r)(1-u)(1-v)\right] }{ (\ln \left[ \tanh (r(0))\tanh (r)(1-u)(1-v)\right] )^2+\theta^2}   \, du dv ,
    \end{align*} 
    which completes the proof.
\end{proof}

\subsection{Windings on the anti de-Sitter space}
We parametrize the anti de Sitter space  
\[
\mathbf{AdS}^{2n+1}=\{ z \in \mathbb{C}^{n+1} | \, | z_1|^2+\cdots+|z_n|^2 -|z_{n+1}|^2=-1 \}
\]
using the coordinates
\begin{align}\label{invar}
(w,\theta)\longrightarrow \frac{e^{i\theta} }{\sqrt{1-|w|^2}} \left( w,1\right),
\end{align}
where $\theta \in [0,2\pi)$ and $w \in \mathbb{C}H^n$. 

 \begin{definition}
The one-form $d\theta$ on $\mathbf{AdS}^{2n+1}$ will be called the winding form on $\mathbf{AdS}^{2n+1}$.
 \end{definition}

 Note that $d\theta$ is closed but not exact and that the integral of  $d\theta$ along any loop is an integer multiple of $2\pi$.  We also note that since $\mathbf{AdS}^{2n+1}$ is diffeomorphic to $\mathbb{S}^1 \times \mathbb{C}^n$  the first integral de Rham cohomology group of  $\mathbf{AdS}^{2n+1}$ is diffeomorphic to $\mathbb Z$, so that $d\theta$ is one of the two generators of that cohomology.
 
 One of our goals in this section will to compute the exact  distribution of the law of the integral of $d\theta$ along Brownian loops. The Laplace operator for the canonical Lorentz metric on $\mathbf{AdS}^{2n+1}$ is denoted ${\square}_{\mathbf{AdS}^{2n+1}}$. We note that ${\square}_{\mathbf{AdS}^{2n+1}} $ is not an elliptic operator. For $\mu \ge 0$, we consider the \emph{Berger Laplacian}, which is the diffusion operator on $\mathbf{AdS}^{2n+1}$ given
\[
\Delta^\mu_{\mathbf{AdS}^{2n+1}} ={\square}_{\mathbf{AdS}^{2n+1}}+(1+\mu^2) \frac{\partial}{\partial \theta^2} .
 \]
 
For $\mu>0$, this operator is the Laplace-Beltrami operator of a Riemannian metric  and for  $\mu=0$, $\Delta^\mu_{\mathbf{AdS}^{2n+1}}$ is the horizontal Laplacian of the anti de-Sitter fibration,  see \cite[Section 4.2]{Baudoin2022}.

Throughout the section, we fix $\mu \ge 0$ and we denote by $(\beta^\mu(t))_{t \ge 0}$ a diffusion process with generator $\frac{1}{2} \Delta_{\mathbf{AdS}^{2n+1}}^\mu$, which is started from the point with coordinates $\frac{e^{i\theta(0)} }{\sqrt{1-|w(0)|^2}} \left( w(0),1\right) $.
By rotational invariance, we can assume,  without loss of generality, that $\theta(0)=0$.
Note that $(\beta^1(t))_{t \ge 0}$ is a Brownian motion and $(\beta^0(t))_{t \ge 0}$ a horizontal Brownian motion.
From \cite[Theorem 4.2.4]{Baudoin2022} we easily deduce that
\[
\beta^\mu(t)= \frac{e^{i\theta(t)} }{\sqrt{1+|w(t)|^2}} \left( w(t),1\right)
\]
where:
\begin{enumerate}
\item $(w(t))_{t \ge 0}$ is a Brownian motion (in local affine coordinates) on $\mathbb{C}H^n$ started from $w(0)$;
\item $\theta(t)=\mu B(t) + \frac{i}{2}\sum_{j=1}^n \int_0^t \frac{w_j(s) d\overline{w}_j(s)-\overline{w}_j(s) dw_j(s)}{1+|w(s)|^2},
 $ with $(B(t))_{t \ge 0}$ a one-dimensional Brownian motion started from 0, which is independent from $(w(t))_{t \ge 0}$.
\end{enumerate}
We will denote $r(t):=\arctanh |w(t)|$.
As pointed out in Subsection \ref{prelim BM CHn}, the process $(r(t))_{t\geq 0}$ is a Jacobi diffusion with generator $\mathcal{L}^{n-1,0}$. 

We have the following formula for the joint density of the couple $(r(t),\theta(t))$.

\begin{theorem}\label{joint ant de sitter}
Let $t>0$. The joint law of $(r(t),\theta(t))$ is given by
\[
d\mathbb{P}_{r(t),\theta(t)}  (r,\theta) = p_{t}(r,\theta)(\sinh (r))^{2n-1}\cosh (r) drd\theta\textrm{ for } r\ge 0\textrm{ and } \theta \in \mathbb R,
\]
where
\begin{equation}\label{eq8}
p_t(r, \theta)=\frac{\Gamma\left( n+\frac{1}{2}\right)}{\sqrt{\pi} \Gamma (n)} \frac{e^{-\frac{\theta^2}{2(1+\mu^2)t}}}{\sqrt{2\pi t}}\int_{-\infty}^{+\infty}\cos \left(\frac{y\theta}{(1+\mu^2)t}  \right)\frac{q^{n-\frac{1}{2},-\frac{1}{2}}_{t}(r(0), \arcosh (\cosh (r)\cosh (y)))}{\left((\cosh (r)\cosh (y))^2-1\right)^n}e^{\frac{y^2}{2(1+\mu^2)t} }dy.
\end{equation}
Therefore, the conditional law $\mathbb{P}_{\theta (t) \mid r(t)=r}$ of $\theta (t)$ given $r(t)=r$ is given by
    \[
    d\mathbb{P}_{\theta(t) \mid r(t)=r}(\theta)= p_t(r, \theta)\frac{(\sinh (r))^{2n-1}\cosh (r)}{q_t^{n-1,0}(r(0),r)}   d\theta,
    \]
    and for $k \in \mathbb Z$, $t>0$, $\beta = e^{i \theta}\left(\frac{w}{\sqrt{1-|w|^2}},\frac{1}{\sqrt{1-|w|^2}}\right) \in \mathbf{AdS}^{2n+1}$, $w\in \mathbb{C}H^n$, $\theta \in [0,2\pi)$, and $r:=\arctan(|w|)$ we have
\[
\mathbb{P}\left(  \theta(t) =\theta +2\pi k \mid \beta^\mu(t)= \beta\right)= \, \frac{p_t(r, \theta +2\pi k)}{\sum_{m \in \mathbb{Z}} p_{t}(r, \theta +2\pi m)}.
\]
\end{theorem}

\begin{proof}
When $\mu=0$, the formula \eqref{eq8} follows from Proposition 4.2.7 in \cite{Baudoin2022}. Since $(B(t))_{t \ge 0}$ is a Brownian motion independent from $(w(t))_{t \ge 0}$ the general formula follows by convolution in the variable $\theta$ with the density of the Gaussian random variable with mean zero and variance $\mu t$. The remainder of the statement follows by arguments similar to the ones in the proof of Theorem \ref{law winding S2n+1}. We  omit the details for the sake of conciseness.
\end{proof}

\begin{proposition} \label{cor:characteristic_function_winding_number_Sitter}
For $\lambda \in \mathbb{R}$, $r\geq 0$ and $t>0$
\begin{align}\label{cond charac anti}
    \mathbb{E} (e^{i\lambda \theta(t)}\mid r(t)=r) =e^{|\lambda| (n-\frac{1}{2}|\lambda|\mu^2 )t}\left(\frac{\cosh (r(0))}{\cosh (r)}\right)^{|\lambda| } \frac{q_t^{n-1,|\lambda| } (r(0),r)}{q_t^{n-1,0} (r(0),r)} .
\end{align}
Therefore for every $\lambda \in \mathbb{R}$ and $\beta = e^{i \theta}\left(\frac{w}{\sqrt{1-|w|^2}},\frac{1}{\sqrt{1-|w|^2}}\right) \in \mathbf{AdS}^{2n+1}$, one has
\[
  \mathbb{E} (e^{i\lambda \theta(t)}\mid \beta^\mu(t)=\beta) =\frac{\sum_{k \in \mathbb Z} e^{-ik\theta} e^{|\lambda+k| (n-\frac{1}{2}|\lambda+k|\mu^2 )t}\left(\frac{\cosh (r(0))}{\cosh (r)}\right)^{|\lambda+k| }q_t^{n-1,|\lambda+k|} (r(0),r)}{\sum_{k \in \mathbb Z} e^{-ik\theta} e^{|k| (n-\frac{1}{2}|k|\mu^2 )t}\left(\frac{\cosh (r(0))}{\cosh (r)}\right)^{|k|}q_t^{n-1,|k |} (r(0),r)}
\]
and
\[
\lim_{t \to +\infty} \mathbb{E} \left(e^{i\lambda \frac{\theta(t)}{\sqrt{t}}}\mid \beta^\mu(t)=\beta\right) =e^{-\frac{1}{2}\lambda ^2(\mu^2+1) )t} .
\]
Consequently, conditional on $\beta^\mu(t)=\beta$, $\frac{\theta(t)}{\sqrt{t}}$ converges in distribution to a Gaussian random variable $\mathcal{N}(0,1+\mu^2)$, when $t \to +\infty$.
\end{proposition}

\begin{proof}
Formula \eqref{cond charac anti} follows from Theorem 4.2.13. in \cite{Baudoin2022} and the fact that $(B(t))_{t \ge 0}$ is a Brownian motion independent from $(w(t))_{t \ge 0}$.  Using similar arguments as in the proof of Corollary \ref{characteristic function bridge S2n+1} and then Lemma \ref{lemma:asymptotics_general_hyperbolic_heat_kernel} we obtain the stated results.
\end{proof}

\subsection{The index of the Brownian loop on  \texorpdfstring{$\mathbf{SL} (2,\mathbb{R})$}{SL(2, R)}}

To conclude the paper, we discuss and highlight separately the case $n=1$ in the anti-de Sitter fibration setting.  In that case the 3-dimensional  anti-de Sitter space $\mathbf{AdS}^3$  is the quadric defined by
\[
\mathbf{AdS}^3=\lbrace z=(z_1,z_{2})\in \mathbb{C}^{2} | \, | z_1 |^2-|z_2|^2 =-1\rbrace \subset \mathbb{C}^{2} \simeq \mathbb{R}^4 .
\]
It is the ball of radius $-1$  for the Lorentz metric on $\mathbb{R}^4$ with real signature $(2,2)$. Consequently, by restriction, $\mathbf{AdS}^3$ is naturally  equipped with a Lorentz metric with signature $(2,1)$. 

Let us now consider  the real special linear group 
\[
\mathbf{SL}(2,\R)= \left\{ \left(\begin{array}{ll}
a & b \\
c & d
\end{array}\right) \mid a,b,c,d \in \R, \, ad-bc=1 \right\}.
\]

Its Lie algebra $\mathfrak{sl} (2)$ consists of $2 \times 2$ matrices of trace $0$.
A basis of $\mathfrak{sl} (2)$ is formed by the matrices:
 \[
X=\left(
\begin{array}{cc}
~1~ & ~0~ \\
~0~ & -1~
\end{array}
\right)
,\text{ }Y=\left(
\begin{array}{cc}
~0~ & ~1~ \\
~1~ & ~0~
\end{array}
\right)\textrm{, and }
\text{ }Z=\left(
\begin{array}{cc}
~0~ & ~1~ \\
-1~& ~0~
\end{array}
\right)
\]
for which the following commutation relations hold
\begin{align}\label{eq-Liestructure}
[X,Y]=2Z, \quad [X,Z]=2Y, \quad [Y,Z]=-2X.
\end{align}

We denote by $\mathbb X,\mathbb Y,\mathbb Z$ the corresponding left invariant vector fields on 
$\mathbf{SL} (2,\mathbb{R})$ and we consider the left invariant Lorentz metric $h$ on $\mathbf{SL} (2,\mathbb{R})$, with real signature $(2,1)$, that makes $\mathbb X,\mathbb Y,\mathbb Z$ an orthonormal frame, i.e.
\[
h(\mathbb X,\mathbb Y)=h(\mathbb Y,\mathbb Z)=h(\mathbb Z,\mathbb X)=0
\]
and
\[
h(\mathbb X,\mathbb X)=h(\mathbb Y,\mathbb Y)=-h(\mathbb Z,\mathbb Z)=1.
\]
It is easy to check that this Lorentz metric is bi-invariant and  a constant multiple of the Killing form of $\mathbf{SL} (2,\mathbb{R})$. Consider then the mapping $\Psi: \mathbf{AdS}^3 \to \mathbf{SL}(2,\R)$ defined by

\begin{align*}
\Psi (x_1+iy_1,x_2+iy_2) :=\left(\begin{array}{ll}
x_1+x_2 & y_1+y_2 \\
y_1-y_2 & x_2-x_1
\end{array}\right) .
\end{align*}

With $w=\tanh (r) e^{i\phi}$ one has
\begin{align*}
\Psi \left( \frac{w e^{i \theta}}{\sqrt{1-|w|^2}} , \frac{ e^{i \theta}}{\sqrt{1-|w|^2}}  \right) &= \Psi \left( \sinh (r) e^{i(\phi+\theta)} ,\cosh(r) e^{i \theta}  \right) \\
 &=\left(
\begin{array}{cc}
 \cosh(r) \cos (\theta) + \sinh(r)\cos (\phi +\theta) & \cosh(r) \sin (\theta) + \sinh(r)\sin (\phi +\theta)\\
-\cosh(r) \sin (\theta) + \sinh(r)\sin (\phi +\theta)& \cosh(r) \cos (\theta) - \sinh(r)\cos (\phi +\theta)
\end{array}\right) \\
 &=\exp \left(r \cos (\phi )X +r \sin (\phi )Y \right) \exp ( \theta Z),
\end{align*}
which yields that $\Psi$ is an isometry between the Lorentz metrics of $\mathbf{AdS}^3$ and $\mathbf{SL} (2,\mathbb{R})$ after comparing \cite[Theorem 4.2.5]{Baudoin2022} to \cite[Section 2]{Bonnefont2}.
Moreover, this isometry satisfies
\[
\Psi (e^{i\theta} z_1, e^{i\theta} z_2)= \Psi (z_1,  z_2) \exp(\theta Z)
\]
and makes the following diagram commutative
\[
 \begin{tikzcd}
 \mathbf{AdS}^3 \arrow[r, "\varphi"] \arrow[d, "\Psi"'] & \mathbb{C}H^{1}\arrow[d, "g^{-1}"]  \\
 \mathbf{SL} (2,\mathbb{R}) \arrow[r,"f"'] & H^2  
 \end{tikzcd}
\]
where $\varphi :(z_1,z_2) \to \frac{z_1}{z_2}$, $f$ is the indefinite Riemannian  submersion $\left(\begin{array}{ll}
a & b \\
c & d
\end{array}\right) \to \frac{ai+b}{ci+d}$ and $g:H^2 \to \mathbb{C}H^{1}$ is the Cayley transform.

It follows from these observations that the Berger Laplacian  on $ \mathbf{SL} (2,\mathbb{R})$ is given by
\begin{align*}
\Delta^\mu =\mathbb X^2+\mathbb Y^2+\mu^2 \mathbb Z^2.
\end{align*}
 In particular, the generator of Brownian motion on $ \mathbf{SL} (2,\mathbb{R})$ is $\frac{1}{2} \Delta^1$.
 We then notice that the set $\mathbf{SL} (2,\mathbb{R})$ is homeomorphic to $\mathbb{R}^2 \times \mathbb{S}^1$. This allows to define the index of a loop $\gamma$ as 
 \[
 \mathbf{Ind}(\gamma)=\int_\gamma \omega \in \mathbb{Z},
 \]
where $\omega$ is a generator of the integral de Rham cohomology $H^{dR}(\mathbf{SL} (2,\mathbb{R}),\mathbb{Z})\simeq \mathbb{Z}$. We have the following theorem:

\begin{theorem}\label{SL2 case}
Let $\mu \ge 0$. Let $(X^\mu(s))_{0 \le s \le t}$ be a Berger Brownian loop of length $t$, i.e. a diffusion process with generator $\frac{1}{2}\Delta^\mu$ started from the identity and conditioned to come back to the identity at time $t$.
Then, for every $k \in \mathbb Z$,
\[
\mathbb{P}\left(  \mathbf{Ind}(X^\mu[0,t]) =k \right) =C_\mu(t) \, e^{-\frac{4\pi^2k^2}{2(1+\mu^2)t}}  \int_{-\infty}^{+\infty}\cos \left(\frac{2\pi k y}{(1+\mu^2)t}  \right) \frac{y}{\sinh(y)}e^{-\frac{\mu^2y^2}{2(1+\mu^2)t} }dy,
\]
where $C_\mu(t)>0$ is the normalization constant.
In particular, for $\mu=0$,
\[
\mathbb{P}\left(  \mathbf{Ind}(X^0[0,t]) =k \right) =C_0 (t) \,  \frac{e^{-\frac{2\pi^2k^2}{t}}}{1+\cosh \left(\frac{2\pi^2k}{t}\right)}.
\]
\end{theorem}

\begin{proof}
From Theorem \ref{joint ant de sitter} we have
\[
\mathbb{P}\left(  \mathbf{Ind}(X^\mu) =k \right)= \, \frac{p_t(0,2\pi k)}{\sum_{m \in \mathbb{Z}} p_{t}(0, 2\pi m)}
\]
where
\begin{align*}
p_t(0, \theta)&=\frac{\Gamma\left( 1+\frac{1}{2}\right)}{\sqrt{\pi} \Gamma (1)} \frac{e^{-\frac{\theta^2}{2(1+\mu^2)t}}}{\sqrt{2\pi t}}\int_{-\infty}^{+\infty}\cos \left(\frac{y\theta}{(1+\mu^2)t}  \right)\frac{q^{-\frac{1}{2},-\frac{1}{2}}_{t}(0, y)}{(\cosh (y))^2-1}e^{\frac{y^2}{2(1+\mu^2)t} }dy.
\end{align*}
Since $q^{-\frac{1}{2},-\frac{1}{2}}_{t}$ is associated to the heat on the 3-dimensional real hyperbolic space, it is well-known (see for example \cite[Section B.3]{Baudoin2022})  that
\[
q^{-\frac{1}{2},-\frac{1}{2}}_{t}(0, y)=\sqrt{\frac{2}{\pi}}e^{-\frac{t}{2}} \frac{y}{t} e^{-\frac{y^2}{2t}}\sinh(y).
\]
Therefore we have
\[
p_t(0, \theta)=\frac{e^{-\frac{t}{2}}}{2\pi t^{3/2}} e^{-\frac{\theta^2}{2(1+\mu^2)t}} \int_{-\infty}^{+\infty}\cos \left(\frac{y\theta}{(1+\mu^2)t}  \right) \frac{y}{\sinh(y)}e^{-\frac{\mu^2y^2}{2(1+\mu^2)t} }dy
\]
and the conclusion follows.
When $\mu=0$, using residue calculus and the contour
\[
\mathcal{C}_R=\left\{ z \in \mathbb{C} \mid  \mathbf{Im} (z) \ge 0, |z|=R \right\} \cup \left\{ z \in \mathbb{C} \mid  \mathbf{Im} (z) = 0, -R \le \mathbf{Re} (z) \le R \right\}
\]
 when $R \to +\infty$, one can compute

\[
\int_{-\infty}^{+\infty}\cos \left(\frac{y\theta}{t}  \right) \frac{y}{\sinh(y)}dy=\frac{\pi^2}{1+\cosh \left(\pi \frac{\theta}{t} \right)},
\]
from which we deduce the stated result.
\end{proof}

\bibliographystyle{plain}
\bibliography{reference.bib}

\begin{thebibliography}{10}

\bibitem{special}
George~E. Andrews, Richard Askey, and Ranjan Roy.
\newblock {\em Special functions}, volume~71 of {\em Encyclopedia of
  Mathematics and its Applications}.
\newblock Cambridge University Press, Cambridge, 1999.

\bibitem{Baudoin2022}
Fabrice Baudoin, Nizar Demni, and Jing Wang.
\newblock {\em Stochastic areas, Horizontal Brownian Motions, and Hypoelliptic
  Heat Kernels}.
\newblock Tracts in Mathematics. European Mathematical Society, 2024.

\bibitem{MR3719061}
Fabrice Baudoin and Jing Wang.
\newblock Stochastic areas, winding numbers and {H}opf fibrations.
\newblock {\em Probab. Theory Related Fields}, 169(3-4):977--1005, 2017.

\bibitem{Besse2007einstein}
A.L. Besse.
\newblock {\em Einstein Manifolds}.
\newblock Classics in Mathematics. Springer Berlin Heidelberg, 2007.

\bibitem{Bonnefont2}
Michel Bonnefont.
\newblock The subelliptic heat kernels on {${\mathrm SL}(2,\mathbb R)$} and on
  its universal covering {$\widetilde{{\mathrm SL}(2,\mathbb R)}$}: integral
  representations and some functional inequalities.
\newblock {\em Potential Anal.}, 36(2):275--300, 2012.

\bibitem{Nizar1}
Nizar Demni.
\newblock Densities of generalized stochastic areas and windings arising from
  anti--de {S}itter and {H}opf fibrations.
\newblock {\em Indag. Math. (N.S.)}, 31(2):204--222, 2020.

\bibitem{Grafakos}
Loukas Grafakos.
\newblock {\em Classical {F}ourier analysis}, volume 249 of {\em Graduate Texts
  in Mathematics}.
\newblock Springer, New York, third edition, 2014.

\bibitem{Koornwinder1984}
Tom~H. Koornwinder.
\newblock {\em Jacobi Functions and Analysis on Noncompact Semisimple Lie
  Groups}, pages 1--85.
\newblock Springer Netherlands, Dordrecht, 1984.

\bibitem{LeGall}
Jean-Francois Le~Gall.
\newblock Marc yor et les nombres de tours du mouvement brownien in \emph{Marc
  Yor, La passion du mouvement brownien}.
\newblock {\em Gazette des math\'ematiciens}, Num\'ero special:39--54, 2015.

\bibitem{Yor1980LoiDL}
Marc Yor.
\newblock Loi de l'indice du lacet brownien, et distribution de hartman-watson.
\newblock {\em Zeitschrift f{\"u}r Wahrscheinlichkeitstheorie und Verwandte
  Gebiete}, 53:71--95, 1980.

\end{thebibliography}

\end{document}